\documentclass{amsart}

\usepackage{amsmath,amsfonts,amsthm,amssymb,amscd,enumerate}
\usepackage{mathtools}
\mathtoolsset{showonlyrefs}
\mathtoolsset{centercolon}
\usepackage{mathrsfs}
\usepackage{esint}

\newtheorem{thm}{Theorem}[section]
\newtheorem{lem}[thm]{Lemma}
\newtheorem{cor}[thm]{Corollary}
\newtheorem{prop}[thm]{Proposition}
\theoremstyle{definition}
\newtheorem{dfn}[thm]{Definition}
\newtheorem{example}[thm]{Example}
\newtheorem{rmk}[thm]{Remark}

\numberwithin{equation}{section}

\usepackage[draft=false, colorlinks=true]{hyperref}

\usepackage{tikz}

\title[Fourier multipliers in Banach function spaces]{Fourier multipliers in Banach function spaces with UMD concavifications}
\author{Alex Amenta, Emiel Lorist, and Mark Veraar}
\address{Delft Institute of Applied Mathematics \\ Delft University of Technology \\ P.O. Box 5031\\ 2600 GA Delft \\The Netherlands}
\thanks{The authors are supported by the VIDI subsidy 639.032.427 of the Netherlands Organisation for Scientific Research (NWO)}
\email{Amenta@fastmail.fm}
\email{E.Lorist@tudelft.nl}
\email{M.C.Veraar@tudelft.nl}

\newcommand{\CC}{\mathbb{C}}
\newcommand{\EE}{\mathbb{E}}
\newcommand{\NN}{\mathbb{N}}
\newcommand{\RR}{\mathbb{R}}
\newcommand{\PP}{\mathbb{P}}
\newcommand{\TT}{\mathbb{T}}
\newcommand{\ZZ}{\mathbb{Z}}
\newcommand{\FF}{\mathscr{F}}

\newcommand{\Sch}{\mathcal{S}}
\newcommand{\calL}{\mc{L}}

\newcommand{\UMD}{\operatorname{UMD}}
\newcommand{\LPR}{\operatorname{LPR}}

\newcommand{\ind}{{{{\bf 1}}}}
\newcommand{\loc}{\operatorname{loc}}
\newcommand{\mb}{\mathbf}
\newcommand{\mc}{\mathcal}

\newcommand{\spn}{\operatorname{span}}
\newcommand{\map}[3]{#1 \colon #2 \rightarrow #3}

\newcommand{\dd}{\hspace{2pt}\mathrm{d}}

\newcommand{\el}[2]{\ell^{#1}(\ell^{#2})} 
\newcommand{\elR}[3]{[#3]_{\ell^{#1}(\ell^{#2})}} 
\newcommand{\elRone}[2]{[#2]_{\ell^{#1}}}
\newcommand{\elRr}[1]{[#1]_{\mc{R}}}
\newcommand{\wh}{\widehat}
\newcommand{\wt}{\widetilde}

\newcommand{\inc}{\mb{\phi}}

\DeclarePairedDelimiter\abs{\lvert}{\rvert}
\DeclarePairedDelimiter\brac[]
\DeclarePairedDelimiter\cbrace\{\}
\DeclarePairedDelimiter\ha()

\DeclarePairedDelimiter\nrm{\lVert}{\rVert}

\newcommand{\nrms}[1]{\Bigl\|#1\Bigr\|}
\newcommand{\abss}[1]{\Bigl|#1\Bigr|}
\newcommand{\has}[1]{\Bigl(#1\Bigr)}
\newcommand{\cbraces}[1]{\Bigl\{#1\Bigr\}}

\newcommand {\Cs}{{\mathscr C}}

\allowdisplaybreaks

\begin{document}
\begin{abstract}
  We prove various extensions of the Coifman--Rubio de Francia--Semmes multiplier theorem to operator-valued multipliers on Banach function spaces.
  Our results involve a new boundedness condition on sets of operators which we call $\el{r}{s}$-boundedness, which implies $\mc{R}$-boundedness in many cases.
  The proofs are based on new Littlewood--Paley--Rubio de Francia-type estimates in Banach function spaces which were recently obtained by the authors.
\end{abstract}




\maketitle

\section{Introduction}

In \cite{jR85} Rubio de Francia proved a surprising extension of the classical Littlewood--Paley square function estimate: for all $p \in [2,\infty)$ there exists a constant $C_p > 0$ such that for \emph{any} collection $\mc{I}$ of mutually disjoint intervals in $\RR$, the estimate
\begin{equation}\label{eqn:LPRest}
  \bigg\| \big( \sum_{I \in \mc{I}} |S_I f|^2 \big)^{1/2} \bigg\|_{L^p(\RR)} \leq C_p \| f \|_{L^p(\RR)}
\end{equation}
holds for all Schwartz functions $f \in \mc{S}(\RR)$, where $S_I$ is the Fourier projection onto $I$.
As a consequence, in \cite{CRS88} Coifman, Rubio de Francia, and Semmes showed that if $p \in (1,\infty)$ and $\frac{1}{s}>\big|\frac{1}{p}-\frac12\big|$, then every $\map{m}{\RR}{\CC}$ of bounded $s$-variation uniformly on dyadic intervals induces a bounded Fourier multiplier $T_m$ on $L^p(\RR)$.
This is analogous to the situation for the Marcinkiewicz multiplier theorem (the $s=1$ case of the Coifman--Rubio de Francia--Semmes theorem), which follows from the classical Littlewood--Paley theorem.

Consider a Banach space $X$.
We are interested in analogues of the results above for operator-valued multipliers on $X$-valued functions; that is, for multipliers $\map{m}{\RR}{\mc{L}_b(X)}$, where $\mc{L}_b(X)$ denotes the space of bounded linear operators on $X$, and where we consider a natural extension of the Fourier transform which acts on $X$-valued functions.
A necessary condition for boundedness of the Fourier multiplier $T_m$ on some Bochner space $L^p(\RR;X)$ is that the range $m(\RR)$ is \emph{$\mc{R}$-bounded} (see Remark \ref{rem:Rnecessity}).
$\mc{R}$-boundedness is a probabilistic strengthening of uniform boundedness which holds automatically for bounded scalar-valued multipliers.
Following the breakthrough papers \cite{CPSW00, We01} there has been an extensive study of operator-valued multiplier theory, in which $\mc{R}$-boundedness techniques are central.
For example, Marcinkiewicz-type theorems were obtained in \cite{AreBu02, BerGil94, Bou86, CPSW00, HHN02,StrWeis08, We01}.
We refer to \cite{HNVW16} for a more detailed historical description.

An operator-valued analogue of the Coifman--Rubio de Francia--Semmes theorem was obtained in \cite{HP06}.
There the Banach space $X$ was assumed to satisfy the so-called $\LPR_{p}$ \emph{(Littlewood--Paley--Rubio de Francia)} property, which was previously studied in \cite{BGT03,GilTor04,HP06,HTY09,PSX12}.
This is a generalisation of the square function estimate \eqref{eqn:LPRest} which may be formulated for all Banach spaces, but which may not hold.
Naturally, $\mc{R}$-boundedness assumptions play a role in the results of \cite{HP06}.
In \cite{ALV1} we proved a range of Littlewood--Paley--Rubio de Francia-type estimates for \emph{Banach function spaces}, including the $\LPR_p$ property, under assumptions involving the $\UMD$ property and convexity (generalising a key result of \cite{PSX12}).
The main goal of this paper is to prove Coifman--Rubio de Francia--Semmes type results for such Banach function spaces.

The following multiplier theorem is the fundamental result of this paper.
Let $\Delta = \{ \pm[2^k, 2^{k+1}), k \in \ZZ\}$ denote the standard dyadic partition of $\RR$.
Let $X$ and $Y$ be Banach function spaces, and for a set of bounded linear operators $\mc{T} \subset \mc{L}_b(X,Y)$ let $V^s(\Delta;\mc{T})$ denote the space of functions $\map{m}{\RR}{\spn(\mc{T})}$ with bounded $s$-variation uniformly on dyadic intervals $J\in \Delta$, measured with respect to the Minkowski norm on $\spn(\mc{T})$ (see below Definition \ref{def:svariation}).
Denote the $q$-concavification of a Banach function space $X$ by $X^q$ (see Section \ref{subs:UMD}).

\begin{thm}\label{thm:main-multintro}
  Let $q \in (1,2]$, $p \in (q, \infty)$, $s\in [1, q)$, and let $w$ be a weight in the Muckenhoupt class $A_{p/q}$.
  Let $X$ and $Y$ be Banach function spaces such that $X^q$ and $Y$ have the $\UMD$ property.
  Let $\mc{T} \subset \mc{L}_b(X,Y)$ be absolutely convex and $\el{2}{q^\prime}$-bounded, and suppose that $m \in V^s(\Delta;\mc{T})$.
  Then the Fourier multiplier $T_m$ is bounded from $L^p(w;X)$ to $L^p(w;Y)$.
\end{thm}

This is proven as part of Theorem \ref{thm:mult-s-var}.
The assumptions on $X$ imply Littlewood--Paley--Rubio de Francia-type estimates that are used in the proof.
In this theorem a condition called `$\el{2}{q^\prime}$-boundedness' appears where one would usually expect an $\mc{R}$-boundedness condition.
This is a new notion which arises naturally from the proof; it turns out to imply $\mc{R}$-boundedness.
We investigate the more general notion of $\el{r}{s}$-boundedness in Section \ref{sec:lrs}.

The case $q=2$ and $w=1$ of Theorem \ref{thm:main-multintro} was considered in \cite[Theorem 2.3]{HP06} for Banach spaces $X = Y$ with the $\LPR_{p}$ property.
Our approach only works for Banach function spaces (and closed subspaces thereof), but these are currently the only known examples of Banach spaces with $\LPR_{p}$.
As the parameter $q$ decreases we assume less of $X$, but more of $\mc{T}$ and $m$.
In Section \ref{sec:multiplier} we prove Theorem \ref{thm:main-multintro}, along with various other extensions and modifications of this result.
In particular we obtain the following improvement of Theorem \ref{thm:main-multintro} for Lebesgue spaces.

\begin{thm}\label{thm:introLr}
  Let $s\in [2, \infty)$.
  Suppose that $\map{m}{\RR}{\mc{L}(L^r(w))}$ for some $r\in (1, \infty)$ and all $w\in A_r(\RR^d)$, and that the following H\"older-type condition is satisfied:
  \begin{equation*}
    \sup_{x\in \RR}\|m(x)\|_{\mc{L}(L^{r}(w))}+ \sup_{J \in \Delta} |J|^{\frac1s} [m]_{C^{1/s}(J;\mc{L}_b(L^{r}(w)))} \leq \inc_r ([w]_{A_{r}}).
  \end{equation*}
Then the Fourier multiplier $T_m$ is bounded on $L^p(\RR;L^r(\RR^d))$ in each of the following cases:
\begin{enumerate}[(i)]
\item $r\in [2, \infty)$ and $\frac1s>\max\big\{\frac12-\frac1p,\frac12-\frac1r, \frac1p-\frac1r\}$,
\item $r\in (1, 2]$ and $\frac1s>\max\big\{ \frac1p-\frac12,\frac1r-\frac12, \frac1r-\frac1p\}$.
\end{enumerate}
\end{thm}

Here $\inc_r([w]_{A_r})$ denotes an unspecified non-decreasing function of the Muckenhoupt characteristic $[w]_{A_r}$.
The result follows from the combination of Proposition \ref{prop:holder-lrs} and Example \ref{ex:Lr-large-s}.
The H\"older assumption allows for the construction of a suitable set $\mc{T}$ as in Theorem \ref{thm:main-multintro}.
The condition on $s$ becomes less restrictive as the numbers $p$, $r$, and $2$ get closer.
Taking $p=r$ or $r=2$ is particularly illustrative: the condition on $s$ is then $\frac{1}{s} > \big|\frac{1}{p} - \frac{1}{2}\big|$, as in the Coifman--Rubio de Francia--Semmes theorem.
However, even if $p=r$, the operator-valued nature of the symbol $m$ prevents us from simply deducing the boundedness of $T_m$ from the scalar-valued case by a Fubini argument.
Using the same techniques, one could also deduce versions of Theorem \ref{thm:introLr} with Muckenhoupt weights in the $\RR$- and $\RR^d$-variables.

In Section \ref{ssec:int-banach} we present some new Coifman--Rubio de Francia--Semmes-type theorems on $\UMD$ Banach spaces (not just Banach functon spaces) which are complex interpolation spaces between a Hilbert space and a $\UMD$ space.
Typical examples which are not Banach function spaces include the space of Schatten class operators, and more generally non-commutative $L^p$-spaces.
Our results in this context are weaker than those that we obtain for Banach function spaces, but nonetheless they seem to be new even for scalar multipliers.

\subsection*{Overview}
\begin{itemize}
\item In Section \ref{sec:prel} we present some preliminaries on Muckenhoupt weights, $\UMD$ Banach function spaces, and Rubio de Francia extrapolation.
\item In Section \ref{sec:lrs} the notion of $\el{r}{s}$-boundedness of a set of operators is defined and investigated.
\item In Section \ref{sec:VsRs} we discuss the class $V^s$ of functions of bounded $s$-variation, and a related atomic space $R^s$.
\item In Section \ref{sec:multiplier} we present our main results, which are several operator-valued Fourier multiplier theorems.
  We cover results for Hilbert spaces, $\UMD$ Banach function spaces, `intermediate' $\UMD$ Banach function spaces, and general `intermediate' $\UMD$ Banach spaces.
\end{itemize}

\subsection*{Notation}

Throughout the paper we consider complex Banach spaces, but everything works just as well for real Banach spaces.

If $\Omega$ is a measure space (we omit reference to the measure unless it is needed) and $X$ is a Banach space,
we let $L^0(\Omega;X)$ denote the vector space of measurable functions modulo almost-everywhere equality, and we let $\Sigma(\Omega;X)$ denote the vector space of all simple functions $f \colon \Omega \to X$.
When $X = \CC$ we write $L^0(\Omega)$ and $\Sigma(\Omega)$.

For vector spaces $V$ and $W$, $\mc{L}(V, W)$ denotes the vector space of linear operators from $V$ to $W$.
For Banach spaces $X$ and $Y$, $\mc{L}_b(X,Y)$ denotes the bounded linear operators from $X$ to $Y$ and $\|T\|_{\mc{L}(X,Y)}$ the operator norm.

Throughout the paper we write $\inc_{a,b,\ldots}$ to denote a non-decreasing function $[1,\infty) \to [1,\infty)$ which depends only on the parameters $a,b,\ldots$, and which may change from line to line.
Nondecreasing dependence on the Muckenhoupt characteristic of weights is used in applications of extrapolation theorems.
We do not obtain sharp dependence on Muckenhoupt characteristics in our results. In \cite[Appendix A]{ALV1} it is shown that monotone dependence on the Muckenhoupt characteristic can be deduced from a more general estimate in terms of the characteristic.

For $p,q \in [1,\infty]$ and $\theta \in [0,1]$, we define the interpolation exponent $[p,q]_\theta$ by
\begin{equation*}
  \frac{1}{[p,q]_\theta} = \frac{1-\theta}{p} + \frac{\theta}{q}
\end{equation*}
with the interpretation $1/0 := \infty$.
This lets us write interpolation results such as $[L^p,L^q]_\theta = L^{[p,q]_\theta}$ in a pleasing compact form.

Occasionally we will work with $\RR^d$ for a fixed dimension $d \geq 1$.
Implicit constants in estimates will depend on $d$, but we will not state this.

\section{Preliminaries\label{sec:prel}}

\subsection{Muckenhoupt weights\label{subs:muck}}
A locally integrable function $w \in L^1_{\loc}(\RR^d)$ is called a \emph{weight} if it is non-negative almost everywhere.
For $p \in [1,\infty)$ the space $L^p(w) = L^p(\RR^d,w)$ consists of all $f \in L^0(\RR^d)$ such that
\begin{equation*}
  \nrm{f}_{L^p(\RR^d,w)} := \has{\int_{\RR^d}\abs{f(x)}^p w(x) \dd x}^{1/p}< \infty.
\end{equation*}
The \emph{Muckenhoupt $A_p$ class} is the set of all weights $w$ such that
  \begin{equation*}
    [w]_{A_{p}} :=\sup_{B} \frac{1}{\abs{B}} \int_B w(x) \dd x \cdot \has{\frac{1}{\abs{B}} \int_B w(x)^{-\frac{1}{p-1}}}^{p-1} < \infty,
  \end{equation*}
  where the supremum is taken over all balls $B \subset \RR^d$, and where the second factor is replaced by $\nrm{w^{-1}}_{L^\infty(B)}$ when $p=1$. Define $A_\infty = \bigcup_{p \geq 1} A_p$.
  For $1 < p \leq q \leq \infty$ we say that a weight $w$ is in the \emph{$\alpha_{p,q}$ class} if $w^{1-p^\prime} \in A_{p^\prime/q^\prime}$, and we write
  \begin{equation*}
    [w]_{\alpha_{p,q}} := [w^{1-p^\prime}]_{A_{p^\prime/q'}}.
  \end{equation*}
  This class naturally arises in duality arguments.
  The $\alpha_{p,2}$ class is used in \cite{sK09}, where it is denoted by $\alpha_p$.

  We will need the following properties of the $A_p$ classes.
\begin{prop}\label{prop:muckenhoupt}\
\begin{enumerate}[(i)]
  \item \label{it:mw3} The $A_p$ classes are increasing in $p$, with $[w]_{A_q} \geq [w]_{A_p}$ when $1 \leq q \leq p$.
  \item \label{it:mw5} For all $w \in A_p$ with $p \in (1,\infty)$ there is an $\varepsilon>0$ such that $w \in A_{p-\varepsilon}$.
  \item \label{it:mw2} For all $w \in A_p$ with $p \in [1,\infty)$ there is a $\delta>0$ such that $w^{1+\delta} \in A_{p}$.
\end{enumerate}
\end{prop}

For proofs and further details on Muckenhoupt weights see \cite[Chapter 9]{lG09}.

\subsection{\texorpdfstring{The $\UMD$ property}{The UMD property}\label{subs:UMD}}
We say that a Banach space $X$ has the $\UMD$ property if the Hilbert transform extends to a bounded operator on $L^p(\RR;X)$ for all $p \in (1,\infty)$.
This is equivalent to the original definition in terms of martingale differences \cite{Burk83,Bour83}.
For a detailed account of the theory of $\UMD$ spaces we refer the reader to \cite{Burk01} and \cite{HNVW16}.
The ``classical'' reflexive spaces (i.e. the reflexive $L^p$ spaces, Sobolev spaces, Besov spaces, Triebel--Lizorkin spaces and Schatten classes) have the $\UMD$ property.
The $\UMD$ property implies reflexivity, so for example $L^1$ and $L^\infty$ do not have the $\UMD$ property.

Most of our results are stated in terms of Banach function spaces that are $p$-convex for some $p \in (1,\infty)$, and whose $p$-concavifications $X^p$ are also Banach function spaces, where $X^p=\{f: \abs{f}^{1/p} \in X\}$ with norm
\begin{equation*}
  \nrm{x}_{X^p} = \nrm{\abs{x}^{1/p}}^p_X.
\end{equation*}
For an introduction to these notions see \cite[Section 2.1]{ALV1}.
We write `$X^p \in \UMD$' as shorthand notation for `$X^p$ is a Banach space which has the $\UMD$ property'.
If $p \geq 1$ this therefore includes the assumption that $X$ is $p$-convex.
The condition that $X^p \in \UMD$ is open in $p$: in fact, if $X^p \in \UMD$, then there exists $\varepsilon > 0$ such that $X^q \in \UMD$ for all $0 < q < p+\varepsilon$ \cite[Theorem 4]{jR86}.
In particular, $X^p \in \UMD$ for some $p >1$ if and only if $X$ is $\UMD$.

\subsection{Extrapolation}
The following Rubio de Francia-type vector-valued extrapolation result was obtained by the authors in \cite[Theorem 3.2]{ALV1}.

\begin{thm}\label{thm:pair-extrap-p}
	Fix $p_0 \in (0,\infty)$ and let $X$ be a Banach function space over $(\Omega,\mu)$ with $X^{p_0}\in \UMD$.
Suppose that $\mc{F} \subset L^0_+(\RR^d;X) \times L^0_+(\RR^d;X)$ and that for all $p > p_0$, $(f,g) \in \mc{F}$, and $w \in A_{p/p_0}$ we have
  \begin{align*}
    \nrm{f(\cdot,\omega)}_{L^p(w)} &\leq \inc_{p,p_0} ([w]_{A_{p/p_0}}) \nrm{g(\cdot,\omega)}_{L^p(w)} \qquad \mu\text{-a.e. }\omega \in \Omega.
  \intertext{Then for all $p>p_0$, $(f,g) \in \mc{F}$, and $w \in A_{p/p_0}$ we have
  }
    \nrm{f}_{L^p(w;X)} &\leq \inc_{X,p,p_0} ( [w]_{A_{p/p_0}} ) \nrm{g}_{L^p(w;X)}.
  \end{align*}
\end{thm}

This theorem implies the following corollary for operators, which is also proved in \cite{ALV1}, where it is formulated more generally.
For the definition of the extension $\widetilde{T}$ see \cite[Lemma 2.4]{ALV1}.

\begin{thm}\label{thm:op-extrap}
Fix $p_0 \in (0,\infty)$, and let $T \in \mc{L}_b(L^p(w))$ for all $p > p_0$ and $w \in A_{p/p_0}$, with
\begin{equation*}
\nrm{T}_{\mc{L}(L^p(w))} \leq \inc_{p,p_0}([w]_{A_{p/p_0}}).
\end{equation*}
Then for all Banach function spaces $X$ with $X^{p_0}\in \UMD$, the operator $T$ has an extension $\widetilde{T}$ on $L^p(w;X)$ for all $p > p_0$ and $w \in A_{p/p_0}$, with
\begin{equation}\label{eqn:extn-norm-est}
  \| \widetilde{T} \|_{\mc{L}(L^p(w;X))} \leq \inc_{X,p,p_0}([w]_{A_{p/p_0}}).
\end{equation}
\end{thm}

We used these results in \cite{ALV1} to deduce Littlewood--Paley--Rubio de Francia-type estimates, and we use them here to prove $\el{r}{s}$-boundedness of families of operators.

\section{\texorpdfstring{$\el{r}{s}$}{lrs}-boundedness}\label{sec:lrs}

Our operator-valued multiplier theorems involve a new condition on sets of bounded operators $\mc{T} \subset \mc{L}_b(X,Y)$, which we call \emph{$\el{r}{s}$-boundedness}.
This generalises the more familiar notions of $\mc{R}$-boundedness and $\ell^s$-boundedness.
In this section we introduce and explore the concept.

\subsection{Definitions and basic properties}

\begin{dfn} Let $X$ and $Y$ be Banach spaces and $\mc{T} \subset \mc{L}_b(X,Y)$.
\begin{itemize}
  \item Let $(\varepsilon_k)_{k=1}^\infty$ be a Rademacher sequence on a probability space $\Omega$.
  We say that $\mc{T}$ is {\em $\mc{R}$-bounded} if for all finite sequences $(T_j)_{j=1}^n$ in $\mc{T}$ and $(x_j)_{j=1}^n$ in $X$,
  \begin{equation*}
    \nrms{\sum_{k=1}^n \varepsilon_k T_k x_k}_{L^2(\Omega;Y)} \lesssim \nrms{\sum_{k=1}^n \varepsilon_k  x_k}_{L^2(\Omega;X)}.
  \end{equation*}
  The least admissible implicit constant is called the $\mc{R}$-bound of $\mc{T}$, and denoted $\elRr{\mc{T}}$.

\item Suppose that $X$ and $Y$ are Banach function spaces and suppose $s \in [1,\infty]$.
  We say that $\mc{T}$ is {\em $\ell^s$-bounded} if for all finite sequences $(T_j)_{j=1}^n$ in $\mc{T}$ and $(x_j)_{j=1}^n$ in $X$,
  \begin{equation*}
    \nrms{\has{\sum_{k=1}^n \abs{T_k x_k}^s}^{1/s}}_Y \lesssim \nrms{\has{\sum_{k=1}^n \abs{ x_k}^s}^{1/s}}_X.
  \end{equation*}
  The least admissible implicit constant is called the $\ell^s$-bound of $\mc{T}$, and denoted $\elRone{s}{\mc{T}}$.

\end{itemize}
\end{dfn}

For a detailed treatment of $\mc{R}$-boundedness we refer the reader to \cite{HNVW2,KW04}, and for $\ell^s$-boundedness see \cite{KU14, Weis01a}.

\begin{dfn}\label{dfn:lrs}
  Let $X$ and $Y$ be Banach function spaces, $\mc{T} \subset \mc{L}_b(X,Y)$ and $r,s \in [1,\infty]$.
  We say that $\mc{T}$ is {\em $\el{r}{s}$-bounded} if for all finite doubly-indexed sequences $(T_{j,k})_{j,k=1}^{n,m}$ in $\mc{T}$ and $(x_{j,k})_{j,k=1}^{n,m}$ in $X$,
  \begin{equation*}
    \nrms{\has{\sum_{j=1}^n \has{\sum_{k=1}^m \abs{T_{j,k} x_{j,k}}^s}^{r/s}}^{1/r}}_Y \lesssim \nrms{\has{\sum_{j=1}^n \has{\sum_{k=1}^m \abs{x_{j,k}}^s}^{r/s}}^{1/r}}_X.
  \end{equation*}
  The least admissible implicit constant is called the $\el{r}{s}$-bound of $\mc{T}$, and denoted $\elR{r}{s}{\mc{T}}$.
\end{dfn}

For $\mc{R}$- and $\ell^2$-boundedness it suffices to consider \emph{subsets} of $\mc{T}$ in the defining inequality (see \cite{CPSW00, KVW16}).
For $\ell^s$- and $\el{r}{s}$-boundedness with $r,s \neq 2$ this is not the case: one must consider \emph{sequences}, allowing for repeated elements.
A singleton $\{T\}$ can fail to be $\ell^s$-bounded, as the defining estimate may fail for arbitrarily long constant sequences $(T,\ldots,T)$ (see \cite[Example 2.16]{KU14}).
We say that an operator $T \in \mc{L}_b(X,Y)$ is $\ell^s$- or $\el{r}{s}$-bounded if the singleton $\cbrace{T}$ is.

If a set $\mc{T} \subset \mc{L}_b(X,Y)$ is $\mc{R}$-, $\ell^s$-, or $\el{r}{s}$-bounded, then so is its closure in the strong operator topology, and likewise its absolutely convex hull $\operatorname{absco}(\mc{T})$.
This was proven in \cite{KW04} for $\mc{R}$-boundedness and \cite{KU14} for $\ell^s$-boundedness; the proof generalises to $\el{r}{s}$-boundedness.

It is immediate from the definition that $\ell^s$-boundedness and $\el{s}{s}$-boundedness are equivalent.
The following proposition encapsulates a few other connections between $\mc{R}$-, $\ell^r$-, and $\el{r}{s}$-boundedness.
For a thorough discussion on the connection between $\mc{R}$ and $\ell^2$-boundedness we refer to \cite{KVW16}.

\begin{prop}\label{prop:Rlslrsequiv}
  Let $X$ and $Y$ be Banach function spaces and $\mc{T} \subset \mc{L}_b(X,Y)$.
  \begin{enumerate}[(i)]
  \item \label{it:equiv1} If $Y$ is $p$-concave for some $p<\infty$ and $\mc{T}$ is $\mc{R}$-bounded, then $\mc{T}$ is $\ell^2$-bounded with $\elRone{2}{\mc{T}} \lesssim  \, \elRr{\mc{T}}$.
  \item \label{it:equiv2} If $X$ is $p$-concave for some $p<\infty$ and $\mc{T}$ is $\ell^2$-bounded, then $\mc{T}$ is $\mc{R}$-bounded with $\elRr{\mc{T}} \lesssim \, \elRone{2}{\mc{T}}$.
  \item \label{it:equiv3} Let $p,s \in [1,\infty]$. If $X$ is $p$-concave, $Y$ is $p$-convex, and $\mc{T}$ is $\ell^{s}$-bounded, then $\mc{T}$ is $\el{p}{s}$-bounded with $\elR{p}{s}{\mc{T}} \leq \, \elRone{s}{\mc{T}}$.
  \item \label{it:equiv4} Let $r,s \in [1,\infty]$. If $\mc{T}$ is $\el{r}{s}$-bounded, then $\mc{T}$ is $\ell^{r}$- and $\ell^{s}$-bounded with $\elRone{r}{\mc{T}} \leq \elR{r}{s}{\mc{T}}$ and $\elRone{s}{\mc{T}} \leq \elR{r}{s}{\mc{T}}$.
  \end{enumerate}
\end{prop}

\begin{proof}
  Statements \eqref{it:equiv1} and \eqref{it:equiv2} follow from the Khintchine-Maurey inequalities (see \cite[Theorem 1.d.6]{LT79}).
  For \eqref{it:equiv3}, consider doubly-indexed finite sequences $(T_{j,k})_{j,k=1}^{m,n}$ in $\mc{T}$ and $(x_{j,k})_{j,k=1}^{m,n}$ in $X$.
  Then we have
  \begin{align*}
    \nrms{\has{\sum_{j=1}^m\has{\sum_{k=1}^n \abs{T_{j,k} x_{j,k}}^s}^{p/s}}^{1/p}}_Y &\leq \has{\sum_{j=1}^m\nrms{\has{\sum_{k=1}^n \abs{T_{j,k} x_{j,k}}^s}^{1/s}}_X^p}^{1/p}\\
    &\leq \elRone{s}{\mc{T}} \has{\sum_{j=1}^m\nrms{\has{\sum_{k=1}^n \abs{ x_{j,k}}^s}^{1/s}}_Y^p}^{1/p} \\
    & \leq \elRone{s}{\mc{T}}\nrms{\has{\sum_{j=1}^m\has{\sum_{k=1}^n \abs{ x_{j,k}}^s}^{p/s}}^{1/p}}_X,
  \end{align*}
  so $\elR{p}{s}{\mc{T}} \leq \, \elRone{s}{\mc{T}}$.
  Finally, \eqref{it:equiv4} follows by taking one index to be a singleton.
\end{proof}

Proposition \ref{prop:Rlslrsequiv} shows in particular that if $\mc{T}$ is $\el{2}{s}$- or $\el{s}{2}$-bounded for some $s \in [1,\infty]$, then $\mc{T}$ is $\ell^2$-bounded, and hence $\mc{R}$-bounded if $Y$ is $p$-concave for some $p<\infty$.

Consider the situation of Theorem \ref{thm:op-extrap}.
If a family of linear operators $\mc{T}$ satisfies the hypothesis of the theorem uniformly, then the family of extensions $\wt{\mc{T}}$ is automatically $\el{r}{s}$-bounded for $r,s>p_0$.
This observation is a convenient source of $\el{r}{s}$-bounded families.

\begin{prop}\label{prop:unifbd-lsX}
  Fix $p_0\in (1, \infty)$, and suppose that $\mc{T} \subset \mc{L}_b(L^p(w))$ for all $p\in (p_0, \infty)$ and $w\in A_{p/p_0}$.
  In addition suppose that for each $T \in \mc{T}$ and $f\in L^p(w)$,
  \begin{align*}
    \nrm{Tf}_{L^{p}(w)} &\leq \inc_{p_0,p} ([w]_{A_{p/p_0}}) \nrm{f}_{L^{p}(w)}.
  \intertext{Let $X$ be a Banach function space with $X^{p_0}\in \UMD$, and let $\wt{\mc{T}}= \{\wt{T}:T\in \mc{T}\}$ be the set of extensions obtained in Theorem \ref{thm:op-extrap}.
  Then for all $p,r,s \in (p_0,\infty)$ and all $w \in A_{p/p_0}$, $\wt{\mc{T}}$ is $\el{r}{s}$-bounded on $L^p(w;X)$ and}
    \elR{r}{s}{\wt{\mc{T}}} &\leq \inc_{p_0,p,r,s,X} ( [w]_{A_{p/p_0}} ).
  \end{align*}
\end{prop}

\begin{proof}
  Consider doubly-indexed finite sequences $(T_{j,k})_{j,k=1}^{m,n}$ in $\mc{T}$ and $(g_{j,k})_{j,k=1}^{m,n}$ in $\Sigma(\RR^d;X)$.
  Let $\Omega$ be the underlying measure space of $X$, and define
  \begin{equation*}
    F,G \colon \RR^d\times \Omega \times \{1, \ldots, m\}\times \{1, \ldots, n\}\to \RR_+
  \end{equation*}
  by
  \begin{align*}
    F(\cdot, \omega,j,k) = \abs{T_{j,k} g_{j,k}(\cdot, \omega)} \quad \text{and} \quad G(\cdot, \omega, j, k) = \abs{g_{j,k}(\cdot, \omega)}.
  \end{align*}
  Then from the assumption on $\mc{T}$ we see that for all $p>p_0$ and all $w\in A_{p/p_0}$,\begin{equation*}
    \|F(\cdot, \omega,j,k)\|_{L^p(w)}\leq \inc_{p_0,p} ([w]_{A_{p/p_0}}) \|G(\cdot, \omega,j,k)\|_{L^p(w)}.
  \end{equation*}
  Letting $Y := X(\ell^r_m(\ell^s_n))$, it follows from \cite[p.\ 214]{jR86} that $Y^{p_0} = X^{p_0}(\ell^{r/p_0}_m(\ell^{s/p_0}_n))$ is $\UMD$, with $\UMD$ constants independent of $m,n\in \NN$.
  Hence Theorem \ref{thm:pair-extrap-p} implies that for all $p\in (p_0,\infty)$ and $w\in A_{p/p_0}$,
  \begin{equation*}
    \|F\|_{L^p(w;Y)}\leq \inc_{X,p_0,p,r,s} ([w]_{A_{p/p_0}}) \|G\|_{L^p(w;Y)}.
  \end{equation*}
  This, combined with \cite[Lemma 2.4]{ALV1}, implies the claimed result.
\end{proof}

Taking $X$ to be the scalar field $\CC$, so that $X^{p_0} = X$ for any $p_0$, we obtain the following special case. Note that in this case a more direct proof may be given as in \cite[Theorem 2.3]{GLV15}.

\begin{prop}\label{prop:unifbd-ls}
Fix $p_0\in (1, \infty)$, and suppose that $\mc{T} \subset \mc{L}_b(L^p(w))$ for all $p\in (p_0, \infty)$ and $w\in A_{p/p_0}$, and in addition suppose that for all $T \in \mc{T}$ and $f\in L^p(w)$,
  \begin{align*}
    \nrm{Tf}_{L^{p}(w)} &\leq \inc_{p_0,p} ([w]_{A_{p/p_0}}) \nrm{f}_{L^{p}(w)}.
  \intertext{Then for all $p,r,s \in (p_0,\infty)$ and all $w \in A_{p/p_0}$, $\mc{T}$ is $\el{r}{s}$-bounded on $L^p(w)$ and}
   	\elR{r}{s}{\mc{T}} &\leq \inc_{p_0,p,r,s} ( [w]_{A_{p/p_0}} ).
   \end{align*}
 \end{prop}

Duality and interpolation may be used to establish $\el{r}{s}$-boundedness, as shown in the following two propositions.

\begin{prop}\label{prop:lrsdual}
  Let $X,Y$ be Banach function spaces, and let $\mc{T} \subset \mc{L}_b(X,Y)$. Let $r,s \in [1,\infty]$. If $\mc{T}$ is $\el{r}{s}$-bounded, then the adjoint family
  \begin{equation*}
    \mc{T}^* = \cbrace{T^*: T\in \mc{T}} \subset \mc{L}_b(Y^*,X^*)
  \end{equation*}
  is $\el{r'}{s'}$-bounded with $\elR{r'}{s'}{\mc{T^*}} = \elR{r}{s}{\mc{T}}$.
\end{prop}

\begin{proof}
    This follows from the duality relation $X(\ell^r_m(\ell^s_n))^* = X^*(\ell^{r'}_m(\ell^{s'}_n))$ (see \cite[Section 1.d]{LT79}).
\end{proof}

To exploit interpolation we must assume order continuity, which holds automatically for reflexive spaces and thus in particular for $\UMD$ spaces (\cite[Section 2.4]{MeyNie}).

\begin{prop}\label{prop:lrsinterpolation}
  Let $X$ and $Y$ be order continuous Banach function spaces and $\mc{T} \subset \mc{L}_b(X,Y)$.
  Let $r_k,s_k \in [1,\infty]$ for $k=0,1$.
  If $\mc{T}$ is $\el{r_k}{s_k}$-bounded for $k=0,1$, then $\mc{T}$ is $\el{r_\theta}{s_\theta}$-bounded for all $\theta \in (0,1)$, where $r_\theta := [r_0,r_1]_\theta$ and $s_\theta := [s_0,s_1]_\theta$.
  Moreover we have the estimate
    \begin{equation*}
        \elR{r_\theta}{s_\theta}{\mc{T}}\leq \elR{r_0}{s_0}{\mc{T}}^\theta \elR{r_1}{s_1}{\mc{T}}^{1-\theta}\leq \max\{\elR{r_0}{s_0}{\mc{T}}, \elR{r_1}{s_1}{\mc{T}}\}.
    \end{equation*}
\end{prop}

\begin{proof}
  This follows from Calder\'on's theory of complex interpolation for order continuous vector-valued function spaces \cite{aC64}.
\end{proof}

Combining Proposition \ref{prop:Rlslrsequiv}\eqref{it:equiv4}  with Proposition \ref{prop:lrsinterpolation} we deduce the following.

\begin{cor}\label{cor:lrstriangle}
  Let $X$ and $Y$ be order continuous Banach function spaces and $\mc{T} \subset \mc{L}_b(X,Y)$.
  Fix $r,s \in [1,\infty]$ and suppose
  that $\mc{T}$ is $\el{r}{s}$-bounded.
  If
  \begin{equation*}
    r \leq u \leq v \leq s \quad \text{or} \quad  s \leq v \leq u \leq r ,
  \end{equation*}	
  then $\mc{T}$ is $\el{u}{v}$-bounded with $\elR{u}{v}{\mc{T}} \leq \elR{r}{s}{\mc{T}}$.
\end{cor}

To end this section we present a technical lemma on the $\el{r}{s}$-boundedness of the closure of a family of operators on spaces other than that in which the closure was taken.
It is used in our multiplier result for intermediate spaces, where several Lebesgue spaces are used simultaneously.
A similar result can be proved with general order continuous Banach function spaces in place of Lebesgue spaces.

\begin{lem}\label{lem:l2-bdd}
  Let $(\Omega, \rho, \mu)$ be a metric measure space, and assume $\mu$ is finite on bounded sets.
  Let $p\in (1, \infty)$ and $\mc{T}\subset \mc{L}(\Sigma(\Omega), L^0(\Omega))$ be such that $\mc{T}\subset \calL(L^p(\Omega))$ is uniformly bounded and absolutely convex.
  Let $\overline{\mc{T}}$ denote the closure of $\mc{T}$ in $\calL(L^p(\Omega))$.
  Suppose $q \in (1,\infty)$, and let $w$ be a weight on $\Omega$ which is integrable on bounded sets.
  Suppose also that $\mc{T}\subset \calL(L^q(w))$ is $\ell^r(\ell^s)$-bounded for some $r,s\in [1, \infty]$.
  Then $\overline{\mc{T}}$ is $\ell^r(\ell^s)$-bounded on $L^q(w)$ with $\elR{r}{s}{\overline{\mc{T}}} = \elR{r}{s}{\mc{T}}$.
\end{lem}

Note that we take the closure $\overline{\mc{T}}$ of $\mc{T}$ in one space, and then establish $\ell^r(\ell^s)$-boundedness of $\overline{\mc{T}}$ considered as a set of operators on a different space.

\begin{proof}
  Fix $(T_{m,n})_{m=1,n=1}^{M,N}$ in $\overline{\wt{T}}$ and $(f_{m,n})_{m=1,n=1}^{M,N}$ in $L^q(w)$.
  By a density argument we may assume each for each $m,n$ that $f_{m,n}$ is bounded and supported on a bounded subset of $\Omega$, which implies $f_{m,n} \in L^p(\Omega)$.
  For each $m,n$ choose $(T_{m,n}^{(k)})_{k\geq 1}$ in $\mc{T}$ such that $T_{m,n}^{(k)}\to T_{m,n}$ in $\calL(L^p(\Omega))$.
  Then also $T_{m,n}^{(k)}f_{m,n}\to T_{m,n} f_{m,n}$ in $L^p(\Omega)$.
  By passing to subsequences we may suppose that for all $m,n$ we have $T_{m,n}^{(k)} f_{m,n}\to T_{m,n} f_{m,n}$, $\mu$-a.e.
  Therefore, by Fatou's lemma,
\begin{align*}
\nrms{\has{\sum_{m=1}^M \has{\sum_{n=1}^N \abs{T_{m,n} f_{m,n}}^s}^{\frac{r}{s}}}^{\frac{1}{r}}}_{L^q(w)}& \leq \liminf_{k\to \infty} \nrms{\has{\sum_{m=1}^M \has{\sum_{n=1}^N \abs{T_{m,n}^{(k)} f_{m,n}}^s}^{\frac{r}{s}}}^{\frac{1}{r}}}_{L^q(w)}  \\ & \leq \elR{r}{s}{\mc{T}} \nrms{\has{\sum_{m=1}^M \has{\sum_{n=1}^N \abs{f_{m,n}}^s}^{\frac{r}{s}}}^{\frac{1}{r}}}_{L^q(w)},
\end{align*}
with the appropriate adjustment if $r=\infty$ or $s=\infty$. So $\overline{\mc{T}}$ is indeed $\ell^r(\ell^s)$-bounded on $L^q(w)$.
\end{proof}

\subsection{\texorpdfstring{$\el{r}{s}$}{lrs}-boundedness of single operators}
As noted before, a single operator $T \in \mc{L}_b(X,Y)$ can fail to be $\el{r}{s}$-bounded.
For positive operators we have the following result, which is an adaptation of \cite[Lemma 4]{sM96}.

\begin{prop}
  Let $X$ and $Y$ be Banach function spaces and let $P \in \mc{L}_b(X,Y)$ be a positive operator.
  Then $P$ is $\el{r}{s}$-bounded for all $r,s \in [1,\infty]$, and we have the $\el{r}{s}$-bound $\elR{r}{s}{\cbrace{P}} \leq \nrm{P}_{\mc{L}(X.Y)}$.
\end{prop}

\begin{proof}
Let $(x_{j,k})_{j,k=1}^{m,n}$ be a doubly-indexed sequence in $X$, and note that by positivity of $P$ we may take the elements of the sequence to be positive.
By positivity of $P$ we can estimate
\begin{align*}
    \nrms{\has{\sum_{j=1}^m \has{\sum_{k=1}^n \abs{P x_{j,k}}^s}^{r/s}}^{1/r}}_Y
    &= \nrms{\sup_{\nrm{\ha{b_j}}_{\ell^{r'}_m} \leq 1 } \sum_{j=1}^m b_j  \sup_{\nrm{\ha{a_k^j}}_{\ell^{s'}_n} \leq1 } \sum_{k=1}^n a_k^j Px_{j,k}}_Y\\
    &\leq \nrms{P\has{\sup_{\nrm{\ha{b_j}}_{\ell^{r'}_m} \leq 1 } \sum_{j=1}^m b_j  \sup_{\nrm{\ha{a_k^j}}_{\ell^{s'}_n} \leq1 } \sum_{k=1}^n a_k^j x_{j,k}}}_Y\\
    &\leq \nrm{P}_{\mc{L}(X,Y)} \nrms{\has{\sum_{j=1}^m \has{\sum_{k=1}^n \abs{x_{j,k}}^s}^{r/s}}^{1/r}}_X,
  \end{align*}
  so $\elR{r}{s}{\cbrace{P}} \leq \nrm{P}_{\mc{L}(X,Y)}$.
\end{proof}

For an $\ell^1$-bounded operator on a Lebesgye space one has $\el{r}{s}$-boundedness for all $r,s\in [1, \infty]$ (see \cite[Theorem 2.7.2]{HNVW16}).
The result below actually holds with $L^p(\Omega)$ replaced by any Banach lattice $X$ with a Levi norm (see \cite{Buh65} and \cite[Fact 2.5]{Lin14}).
A duality argument implies a similar result for $\ell^\infty$-boundedness.

\begin{prop}
Let $p\in [1, \infty)$ and $T\in \mc{L}(L^p(\Omega))$. If $T$ is $\ell^1$-bounded, then $\{T\}$ is $\el{r}{s}$-bounded for all $r,s\in [1, \infty]$.
\end{prop}

\begin{rmk}
  Even on $L^p$ it can be quite hard to establish the $\el{r}{s}$-boundedness of a single operator.
  By using i.i.d.\ $s$-stable random variables $\xi_1, \ldots, \xi_n\colon\Omega\to \RR$ (see \cite[Section 5]{LT91}), for $p\in (0,s)$ one can linearise the estimate by writing
  \[\Big(\sum_{j=1}^n |Tx_j|^s\Big)^{1/s} = C_{p,s}\Big\|T \sum_{j=1}^n \xi_j x_j\Big\|_{L^p(\Omega)}.\]
By using Fubini's theorem and Minkowski's inequality, one can deduce that any $T\in \calL(L^p)$ is $\el{r}{s}$-bounded if $p\leq r\leq s\leq 2$ or $2\leq s\leq r\leq p$.
  Most of the remaining cases seem to be open (see \cite[Problem 2]{Kwapfactor72} and \cite[Corollary 1.44]{dales2016multi}).
\end{rmk}

\subsection{Non-examples\label{subsect:counter}}

We end this section with two examples to demonstrate that $\el{r}{s}$-boundedness is not just the conjunction of $\ell^r$- and $\ell^s$-boundedness.
Consider the class of kernels
\begin{equation*}
  \mc{K} = \cbrace{k \in L^1(\RR): \abs{k*f} \leq Mf \text{ a.e. for all simple } f\colon\RR \to \RR},
\end{equation*}
where $M$ is the Hardy--Littlewood maximal operator. For $k \in \mc{K}$ and $f \in L^p(\RR)$ with $p \in (1,\infty)$ define an operator $T_k$ by
\begin{align*}
  T_kf(t) &= \int_{\RR} k(t-s)f(s) \dd s,
\end{align*}
and set $\mc{T} = \cbrace{T_k:k \in \mc{K}}$.

\begin{example}\label{ex:Tk}
  Let $p\in (1, \infty)$.
  The family of operators $\mc{T} \subset \mc{L}_b(L^p(\RR))$ defined above is $\ell^s$-bounded for all $s\in [1, \infty]$, but not $\el{1}{s}$- or $\el{\infty}{s}$- bounded for any $s\in (1, \infty)$.
\end{example}

\begin{proof}
  The $\ell^s$-boundedness of $\mc{T}$ for $s \in [1,\infty]$ is proved in \cite[Theorem 4.7]{NVW15}.
  Since $\mc{T} = \mc{T}^*$, Proposition \ref{prop:lrsdual} says that $\el{1}{s}$-boundedness of $\mc{T}$ on $L^p(\RR)$ implies $\el{\infty}{s^\prime}$-bounded\-ness on $L^{p^\prime}(\RR)$, so it suffices to show that $\mc{T}$ is not $\el{\infty}{s}$-bounded on $L^p(\RR)$ for any $s \in (1,\infty)$.
  We follow the proof of \cite[Proposition 8.1]{NVW15}.

  Fix $n \in \NN$ and for  $i,j \in \NN$ define $f_{i,j} \in L^p(\RR)$ by
  \begin{equation*}
    f_{i,j}(t) = \ind_{(0,1]}(t) \ind_{(2^{-j},2^{-j+1}]}(t-(i-1)2^{-n})
  \end{equation*}
  so that
  \begin{equation}\label{eqn:lsinfty-rhs}
    \nrms{\sup_{1\leq i \leq 2^n}\has{\sum_{j=1}^n \abs{f_{i,j}(t)}^s}^{1/s}}_{L^p(\RR)} \leq \nrms{\sup_{1\leq i \leq 2^n}\ind_{(0,1]}}_{L^p(\RR)} = 1.
  \end{equation}

  Next, for $i,j \in \NN$  define
  \begin{equation*}
    k_{i,j}(t) = \frac{1}{2^{-j+2}} \, \ind_{(-2^{-j+1},2^{-j+1})}(t)
  \end{equation*}
  and $T_{i,j} = T_{k_{i,j}}$.
  Then $T_{i,j} \in \mc{T}$, as for any simple function $f$ we have
  \begin{align*}
    \abs{T_{i,j}f(t)} =\abs{k_{i,j} *f (t)} &= \frac{1}{2^{-j+2}} \, \abss{\int_{\RR} \ind_{(-2^{-j+1},2^{-j+1})}(t-\tau)f(\tau)\, d\tau}\\
                                            &= \frac{1}{2^{-j+2}} \, \abss{\int_{t-2^{-j+1}}^{t+2^{-j+1}} f(\tau)\, d\tau} \leq Mf(t).
  \end{align*}
  Furthermore, for any $1 \leq j \leq n$, $t \in (0,1]$ and $1 \leq i \leq 2^{n}$ with  $t\in ((i-1)2^{-n},i2^{-n}]$,
  \begin{align*}
    \abs{T_{i,j}f_{i,j}(t)} &= \frac{1}{2^{-j+2}} \int_{t-2^{-j+1}-(i-1)2^{-n}}^{t+2^{-j+1}-(i-1)2^{-n}} \ind_{(2^{-j},2^{-j+1}]}(\tau)\, d\tau\\
                            &\geq \frac{1}{2^{-j+2}} \int_{2^{-j}}^{2^{-j+1}} \ind_{(2^{-j},2^{-j+1}]}(\tau)\, d\tau = \frac{2^{-j}}{2^{-j+2}} = \frac14.
  \end{align*}
  Therefore
  \begin{align*}
    \nrms{\sup_{1\leq i \leq 2^n} \has{ \sum_{j=1}^n \abs{T_{i,j}f_{i,j}(t)}^s }^{1/s} }_{L^p(\RR)} \geq \nrms{ \has{\frac{n}{4^s}}^{1/s} \ind_{(0,1]}}_{L^p(\RR)} = \frac{n^{1/s}}{4}
  \end{align*}
  which tends to $\infty$ as $n \to \infty$.
  Combining this with \eqref{eqn:lsinfty-rhs} disproves the $\el{\infty}{s}$-boundedness of $\mc{T}$ on $L^{p}(\RR)$.
\end{proof}

The previous example can be modified to construct examples without $\el{2}{s}$-boundedness, by using stochastic integral operators.
For $k \in \mc{K}$ and $f \in L^p(\RR_+)$ with $p \in (2,\infty)$, define
\begin{equation*}
  S_k f(t) := \int_0^t |k(t-s)|^\frac12f(s) \dd W(s),
\end{equation*}
where $W$ is a standard Brownian motion on a probability space $(\Omega,\mc{F},\PP)$.
Define $\mc{S} := \cbrace{S_k:k \in \mc{K}}$.

\begin{example}\label{eg:non-el2r-bdd}
Let $p\in (2, \infty)$. The family of operators $\mc{S}$ from $L^p(\RR_+)$ to $L^p(\RR_+ \times \Omega)$ is $\ell^r$-bounded for all $r\in [2, \infty)$, but not $\el{2}{r}$-bounded for any $r\in (2, \infty)$.
\end{example}

\begin{proof}
  Let $r \in [2,\infty)$ and $X = \ell^r$.
  Take $f \in L^p(\RR_+;X)$ and $k \in L^1(\RR_+;X)$ such that $k_{j} \in \mc{K}$ for all $j \in \NN$.
  By \cite[Corollary 2.10]{NW05} and the Kahane--Khintchine inequalities (see for example \cite{LT91}), we know that
  \begin{equation*}
    \has{\EE\nrms{\int_0^t \abs{k(t-s)}^\frac{1}{2}\abs{f(s)}\dd W(s)}_X^p}^{1/p} \simeq \nrms{\has{\int_0^t \abs{k(t-s)}\abs{f(s)}^2\dd s}^\frac{1}{2}}_X
  \end{equation*}
  for any $t \in \RR_+$.
  This implies that $\mc{S}$ is $\ell^r$-bounded from $L^p(\RR_+)$ to $L^p(\RR_+ \times \Omega)$ if and only if $\mc{T}$ restricted to $\RR_+$ is $\ell^{r/2}$-bounded on $L^{p/2}(\RR_+)$, so $\mc{S}$ is $\ell^r$-bounded for all $r \in [2,\infty)$ by Example \ref{ex:Tk}. Repeating the argument with $X = \ell^2(\ell^r)$, we also get from Example \ref{ex:Tk} that  $\mc{S}$ is not $\el{2}{r}$-bounded for any $r\in (2, \infty)$.
\end{proof}

\section{The function spaces \texorpdfstring{$V^s(\mc{J};Y)$}{Vs(J;Y)}\label{sec:VsRs} and \texorpdfstring{$R^s(\mc{J};Y)$}{Rs(J;Y)} }

The multipliers we consider are members of the space of functions of bounded $s$-variation, which we denote by $V^s(\mc{J},Y)$ for $s \geq 1$.
This space contains the class of $1/s$-H\"older continuous functions.
In our arguments we will also use the atomic function space $R^s(\mc{J},Y)$, which was introduced in the scalar case in \cite{CRS88}.

\begin{dfn}\ \label{def:svariation}
  \begin{enumerate}[(i)]
  \item Let $Y$ be a Banach space, $J =[J_-,J_+] \subset \RR$ a bounded interval and $s \in [1,\infty)$. A function $\map{f}{\RR}{Y}$ is said to be of \emph{bounded $s$-variation on $J$}, or $f \in V^s(J;Y)$, if
    \begin{equation*}
      \nrm{f}_{V^s(J;Y)} := \|f\|_\infty +  [f]_{{V_s(J;Y)}} <\infty,
    \end{equation*}
    where
    \begin{equation*}
      [f]_{V_s(J;Y)}  := \sup_{J_- = t_0 <  \cdots < t_N = J_+} \has{\sum_{i=1}^N \nrm{f(t_{i-1}) - f(t_i)}_Y^s}^{1/s}.
    \end{equation*}
    Furthermore we define $V^\infty(J;Y) = L^\infty(J;Y)$.

  \item When $\mc{J}$ is a collection of mutually disjoint bounded intervals in $\RR$, the space $V^s(\mc{J};Y) \subset L^\infty(\RR;Y)$ consists of all $f \in  L^\infty(\RR;Y)$ such that
    \begin{equation*}
      \nrm{f}_{V^s(\mc{J};Y)} := \sup_{J \in \mc{J}} \nrm{f|_J}_{V^s(J;Y)} < \infty.
    \end{equation*}
    If $\mc{J} = (J_k)_{k \in \NN}$ is ordered, we define $V^s_0(\mc{J};Y) \subset V^s(\mc{J};Y)$ to be the closed subspace consisting of $f \in V^s(\mc{J};Y)$ with $\lim_{k \to \infty} \nrm{f|_{J_k}}_{V^s(J;Y)} = 0$.
\end{enumerate}
\end{dfn}

Clearly $V^s(\mc{J};Y) \hookrightarrow V^t(\mc{J};Y)$ contractively when $1 \leq s \leq t \leq \infty$, and $V^s(\mc{J};Y)$ is complete when $Y$ is complete.

In our applications the space $Y$ is usually the span of a bounded and absolutely convex subset $B$ of a normed space $Z$  (i.e. a \emph{disc} in $Z$), equipped with the Minkowski norm
\begin{equation*}
  \nrm{x}_B := \inf\{ \lambda > 0 : \tfrac{x}{\lambda} \in B \},
\end{equation*}
and we write
$ V^s(\mc{J};B) := V^s(\mc{J};\spn B)$.
Clearly $\|x\|_{Z}\leq C_B \|x\|_{B}$ for $x\in Y$.
If the Minkowski norm on $\spn B$ is complete, then $B$ is called a \emph{Banach disc}.
  If $Z$ is a Banach space and $B$ is closed, then $B$ is a Banach disc \cite[Proposition 5.1.6]{PCB87}, but this is not a necessary condition \cite[Proposition 3.2.21]{PCB87}.

\begin{dfn}\label{def:Rspaces}\
  \begin{enumerate}[(i)]
  \item
    Let $Y$ be a normed space, $J \subset \RR$ a bounded interval, and $s \in [1,\infty)$.
    Say that a function $\map{a}{J}{Y}$ is an \emph{$R^s(J;Y)$-atom}, written $a \in R_{\text{at}}^s(J;Y)$, if there exists a set $\mc{I}$ of mutually disjoint subintervals of $J$ and a set of vectors $(c_I)_{I \in \mc{I}} \subset Y$ such that
    \begin{equation*}
      a = \sum_{I \in \mc{I}} c_I \mb{1}_I \quad \text{and} \quad \has{ \sum_{I \in \mc{I}}  \nrm{c_I}_Y^s }^{1/s} \leq 1.
    \end{equation*}	
    Define $R^s(J;Y) \subset L^\infty(J;Y)$ by
    \begin{equation*}
      R^s(J;Y) := \cbraces{ f \in L^\infty(J;Y) : f = \sum_{k=1}^\infty \lambda_k a_k,  (\lambda_k) \in \ell^1, (a_k) \subset R_{\text{at}}^s(J;Y)},
    \end{equation*}
    where the series $f = \sum_{k=1}^\infty \lambda_k a_k$ converges in $L^\infty(J;Y)$.
    Define a norm on $R^s(J;Y)$ by
    \begin{equation*}
      \nrm{f}_{R^s(J;Y)} := \inf\cbraces{ \nrm{\lambda_k}_{\ell^1} : \text{$f = \sum_{k=1}^\infty \lambda_k a_k$ as above} }.
    \end{equation*}
    Furthermore we define $R^\infty(J;Y) := L^\infty(J;Y)$.

  \item
    When $\mc{J}$ is a collection of mutually disjoint bounded intervals in $\RR$, the space $R^s(\mc{J};Y) \subset L^\infty(\RR;Y)$ consists of all $f \in  L^\infty(\RR;Y)$ such that
    \begin{equation*}
      \nrm{f}_{R^s(\mc{J};Y)} := \sup_{J \in \mc{J}} \nrm{f|_J}_{R^s(J;Y)} < \infty.
    \end{equation*}
    If $\mc{J} = (J_k)_{k \in \NN}$ is ordered, we define $R^s_0(\mc{J};Y) \subset R^s(\mc{J};Y)$ to be the closed subspace consisting of $f \in R^s(\mc{J};Y)$ with $\lim_{k \to \infty} \nrm{f|_{J_k}}_{R^s(J_k;Y)} = 0$.
  \end{enumerate}
\end{dfn}
Clearly $R^s(\mc{J};Y) \hookrightarrow R^t(\mc{J};Y)$ contractively when $1 \leq s \leq t \leq \infty$, and $R^s(\mc{J};Y)$ is complete when $Y$ is complete. As with the classes $V^s$, when $B$ is a disc in a normed space $Z$, we put the Minkowski norm on the linear span of $B$ and write $R^s(\mc{J};B) := R^s(\mc{J};\spn B)$.

For $\alpha\in (0,1]$ and an interval $J \subset \RR$ we let $C^\alpha(J;Y)$ denote the space of $\alpha$-H\"older continuous functions with $\nrm{f}_{C^\alpha(J;Y)} = \max\{ \nrm{f}_{\infty}, [f]_{C^\alpha(J;Y)} \}$, where
\begin{equation*}
  [f]_{C^\alpha(J;Y)} := \sup_{x,y \in J} \frac{\nrm{f(x) - f(y)}_Y}{|x-y|^\alpha}.
\end{equation*}

\begin{lem}\label{lem:embeddingVHR}
Let $s\in [1, \infty)$, let $Y$ be a Banach space and fix a bounded interval $J\subset \RR$.
\begin{enumerate}[(i)]
\item If $q\in (s, \infty)$, then $R^{s}(J;Y) \subset V^s(J;Y) \subset R^q(J;Y)$ and for all $f\in L^\infty(J;Y)$ we have
  \begin{equation*}
    \|f\|_{R^q(J;Y)}\lesssim_{q,s}  \|f\|_{V^s(J;Y)} \lesssim \|f\|_{R^s(J;Y)} .
  \end{equation*}
\item We have $C^{1/s}(J;Y)\subset V^s(J;Y)$, and for all $f\in V^s(J;Y)$,
  \begin{equation*}
    \|f\|_{V^s(J;Y)}\leq \|f\|_{\infty} + |J|^{1/s} [f]_{C^{1/s}(J;Y)}.
  \end{equation*}
\end{enumerate}
\end{lem}
\begin{proof}
For part (i) we note that both $R^{s}(J;Y) \subset V^s(J;Y)$ and the second norm estimate follow directly from the fact that for any atom $a \in R^s_\text{at}(J;Y)$ with
\begin{equation*}
  a = \sum_{I \in \mc{I}} c_I \mb{1}_I
\end{equation*}
we have by Minkowski's inequality that
\begin{align*}
  \nrm{a}_{V^s(J;Y)} \leq \sup_{I \in \mc{I}} \nrm{c_I}_Y + \has{\sum_{\substack{I,J \in \mc{I}\\I\neq J}} \nrm{c_I-c_J}^s}^{1/s}\leq 1+ 2 \has{\sum_{I \in \mc{I}} \nrm{c_I}^s}^{1/s} \leq 3.
\end{align*}
The embedding $V^s(J;Y) \subset R^q(J;Y)$ with the first norm estimate is shown in \cite[Lemme 2]{CRS88} for scalar functions, and the argument extends to the general case.
Part (ii) is straightforward to check.
\end{proof}

We end this section with complex interpolation containments for the $V^s$- and $R^s$-classes.
It is an open problem whether complex interpolation of the $V^s$-classes as below can be proved with $\varepsilon = 0$ (see \cite[Chapter 12]{gP16}).
It is also not clear whether converse inclusions hold, but since we don't need them we leave the question open.

\begin{thm}\label{thm:VR-interpoln}
  Suppose $1 \leq q_0 \leq q_1\leq \infty$, $\theta \in (0,1)$, $\varepsilon>0$ and let $Y$ be a Banach space.
  Then for all bounded intervals $J \subset \RR$ we have continuous inclusions
  \begin{align}
    \label{eq:Vs-single} V^{[q_0,q_1]_\theta-\varepsilon}(J;Y) &\hookrightarrow [V^{q_0}(J;Y), V^{q_1}(J;Y)]_\theta, \\
   \label{eq:Rs-single} R^{[q_0,q_1]_\theta}(J;Y) &\hookrightarrow [R^{q_0}(J;Y), R^{q_1}(J;Y)]_\theta, \qquad q_1 \neq \infty.
  \end{align}
  Furthermore, if $\mc{J} = (J_k)_{k \in \NN}$ is an ordered collection of mutually disjoint bounded intervals in $\RR$, then we have continuous inclusions
  \begin{align}
    \label{eq:Vs-multi}V_0^{[q_0,q_1]_\theta-\varepsilon}(\mc{J};Y) &\hookrightarrow [V_0^{q_0}(\mc{J};Y), V_0^{q_1}(\mc{J};Y)]_\theta\\
     \label{eq:Rs-multi} R_0^{[q_0,q_1]_\theta}(\mc{J};Y) &\hookrightarrow [R_0^{q_0}(\mc{J};Y), R_0^{q_1}(\mc{J};Y)]_\theta, \qquad q_1 \neq \infty.
  \end{align}
\end{thm}

\begin{proof}
  For $q_0=1$ and $q_1=\infty$ we have \eqref{eq:Vs-single} by applying subsequently \cite[Lemma 12.11]{gP16}, \cite[Theorem 3.4.1]{BL76}, and \cite[Theorem 4.7.1]{BL76},
  \begin{align*}
    V^{[q_0,q_1]_\theta-\varepsilon}(J;Y) &\hookrightarrow \ha*{V^1(J;Y), L^\infty(J;Y)}_{\theta_{\varepsilon},\infty}   \\
     &\hookrightarrow \ha*{V^1(J;Y), L^\infty(J;Y)}_{\theta,1} \\ &\hookrightarrow \brac*{V^1(J;Y), L^\infty(J;Y)}_{\theta}
  \end{align*}
  with
  \begin{equation*}
    \theta_\varepsilon = 1 - \frac{1}{\frac{1}{1-\theta}-\varepsilon} < \theta.
  \end{equation*}
  The intermediate cases follow from the reiteration theorem for complex interpolation \cite[Theorem 4.6.1]{BL76}.

\bigskip

In the remainder of the proof we will need the following notation: when $\mc{I}_k$ is a collection of intervals for each $k \in \NN$ and $I \in \mc{I}_k$, let $\pi_{I,k}$ denote the canonical projection $\ell^\infty(\mc{I}_k;Y) \to Y$. We abbreviate Banach couples $(X_0,X_1)$ by $X_\bullet$, and use this shorthand for expressions like
\begin{equation*}
  [\ell^{p_\bullet}(\NN;X)]_\theta = [\ell^{p_0}(\NN;X), \ell^{p_1}(\NN;X)]_\theta.
\end{equation*} We let $\mc{F}(X_\bullet)$ denote the space of bounded analytic functions from the closed strip $\overline{S} := \{z \in \CC : \Re z \in [0,1]\}$ to the sum $X_0 + X_1$ whose restrictions to the sets $\{z \in \CC : \Re z = 0\}$ and $\{z \in \CC : \Re z = 1\}$ map continuously into $X_0$ and $X_1$ respectively, equipped with the norm
\begin{equation*}
  \nrm{F}_{\mc{F}(X_\bullet)} := \max\left( \sup_{t \in \RR} \nrm{F(it)}_{X_0}, \sup_{t \in \RR} \nrm{F(1 + it)}_{X_1} \right)
\end{equation*}
as in \cite[\textsection 4.1]{BL76}.

\bigskip

        For \eqref{eq:Rs-single} let $1 \leq q_0 \leq q_1\leq \infty$ and write $q_\theta := [q_0,q_1]_\theta$ for brevity.  Suppose $f \in R^{q_\theta}(J;Y)$, with atomic decomposition
        \begin{equation*}
          f = \sum_{k=1}^\infty \lambda_k a_k =  \sum_{k=1}^\infty \lambda_k \sum_{I \in \mc{I}_k} \mb{1}_I \pi_{I,k} (c_k),
        \end{equation*}
        where $c_k \in \ell^{q_\theta}(\mc{I}_k;Y)$ for each $k \in \NN$.

        Let $\varepsilon > 0$.
        For each $k \in \NN$ we have $\ell^{q_\theta}(\mc{I}_k;Y) = [\ell^{q_\bullet}(\mc{I}_k;Y)]_\theta$ with equal norms \cite[Theorem 1.18.1]{hT78}, hence there exists a function $C_k \in \mc{F}(\ell^{q_\bullet}(\mc{I}_k;Y))$ with $C_k(\theta) = c_k$ and $\|C_k\|_{\mc{F}(\ell^{q_\bullet}(\mc{I}_k;Y))} \leq (1+\varepsilon)\|c_k\|_{\ell^{q_\theta}(\mc{I}_k;Y)} \leq 1 + \varepsilon$.
        For all $z \in \overline{S}$ and $t \in J$, define
        \begin{equation*}
          A_{k}(z)(t) := \sum_{I \in \mc{I}_k} \mb{1}_I(t) \pi_{I,k}(C_{k}(z)),
        \end{equation*}
        noting that for each $t$ there is at most one non-zero term in the sum.
        It follows from $\|C_k\|_{\mc{F}(\ell^{q_\bullet}(\mc{I}_j;Y))} \leq 1 + \varepsilon$ that
        $\nrm{A_{k}}_{\mc{F}(R^{q_\bullet}(J;Y))} \leq 1 + \varepsilon$ for all $z \in \overline{S}$.

        We will show that each $\map{A_{k}}{S}{R^{q_0}(J;Y) + R^{q_1}(J;Y)}$ is analytic on $S$, using that $R^{q_0}(J;Y) + R^{q_1}(J;Y) = R^{q_1}(J;Y)$ and $\ell^{q_0}(\mc{I}_k;Y) + \ell^{q_1}(\mc{I}_k;Y) = \ell^{q_1}(\mc{I}_k;Y)$.
        Fix $z_0 \in S$.
        Since $C_{k}$ is analytic with values in $\ell^{q_1}(\mc{I}_k;Y)$, there exists a Taylor expansion
        \begin{equation*}
          C_{k}(z) = \sum_{n=0}^\infty (z-z_0)^n \beta_{k,n}
        \end{equation*}
        for $z$ in a neigbourhood of $z_0$, where $(\beta_{k,n})_{n=0}^\infty \subset \ell^{q_0}(\mc{I}_k;Y)$ is a bounded sequence.
        Thus for such $z$ we have
        \begin{align*}
          A_{k}(z)
          &= \sum_{I \in \mc{I}_k} \mb{1}_I \pi_{I,k}(C_{k}(z))
          = \sum_{n=0}^\infty (z-z_0)^n \sum_{I \in \mc{I}_k} \mb{1}_{I} \pi_{I,k} (\beta_{k,n})
          =: \sum_{n=0}^\infty (z-z_0)^n \gamma_{k,n}
        \end{align*}
        using the mutual disjointness of $\mc{I}_k$ to interchange the sums.
        The functions $\gamma_{k,n}$ are in $R^{q_1}(J;Y)$ as we can write
        \begin{equation*}
          \nrm{\gamma_{k,n}}_{R^{q_1}(J;Y)}
          = \nrms{\sum_{I \in \mc{I}_k} \mb{1}_I \pi_{I,k} (\beta_{k,n})}_{R^{q_1}(J;Y)} \\
          \leq \nrm{\beta_{k,n}}_{\ell^{q_1}(\mc{I}_k;Y)} < \infty.
        \end{equation*}
        Similarly we can show that each $\map{A_{k}}{\overline{S}}{R^{q_1}(J;Y)}$ is continuous.

        Now for $z \in \overline{S}$ and $j \in \NN$ define
        \begin{equation*}
          F(z) := \sum_{k=1}^\infty \lambda_k A_{k}(z).
        \end{equation*}
        Since the functions $\map{A_{k}}{S}{R^{q_0}(J;Y) + R^{q_1}(J;Y)}$ are bounded uniformly in $k$, continuous on $\overline{S}$, and analytic on $S$, and since $\lambda \in \ell^1(\NN)$, and each $A_{k}$ maps into $R^{q_0}(J;Y) + R^{q_1}(J;Y)$, we find that each $F \in \mc{F}(R^{q_\bullet}(J;Y))$.
        Furthermore we have
        \begin{equation*}
          F(\theta) = \sum_{k=1}^\infty \lambda_k A_k(\theta) = \sum_{k=1}^\infty \lambda_k \sum_{I \in \mc{I}_k} \mb{1}_I \pi_{I,k}(C_k(\theta)) = f
        \end{equation*}
        and
        \begin{equation*}
          \nrm{F}_{\mc{F}(R^{q_\bullet}(J;Y))} \leq \nrm{\lambda_k}_{\ell^1(\NN)} \sup_{k \in \NN} \nrm{A_k}_{\mc{F}(R^{q_\bullet}(J;Y))} \leq (1+\varepsilon) \nrm{\lambda_k}_{\ell^1(\NN)}.
        \end{equation*}
        Since $\varepsilon > 0$ was arbitrary, taking the infimum over all atomic decompositions of $f$ and all possible $F \in \mc{F}(R^{q_\bullet}(J;Y))$ with $F(\theta) = f$ completes the proof.

\bigskip

  Now consider a collection $\mc{J}$ of mutually disjoint bounded intervals in $\RR$.
  We will only prove \eqref{eq:Vs-multi}, as the proof of \eqref{eq:Rs-multi} is similar.
  We introduce the following notation: if $J = [J_-, J_+) \subset \RR$ is a bounded interval and $f \in L^0(J;Y)$, we let $f_J \in L^0([0,1);Y)$ be the function
  \begin{equation*}
    f_J(x) := f((J_+ - J_-)x + J_+) \qquad x \in [0,1).
  \end{equation*}
  Then for each $s \in [1,\infty]$ the map $\map{\tau_J}{V^s(J;Y)}{V^s([0,1);Y)}$ defined by
  $\tau_J(f) := f_J $
  is an isometry. Consequently we can write
  \begin{equation*}
    \nrm{f}_{V^s(\mc{J};Y)} = \sup_{J \in \mc{J}} \nrm{f|_J}_{V^s(J;Y)} = \sup_{J \in \mc{J}} \nrm{\tau_J(f|_J)}_{V^s([0,1);Y)},
  \end{equation*}
  and therefore the map $\map{\Phi}{V_0^s(\mc{J};Y)}{c_0(\mc{J};V^s([0,1);Y))}$ defined by
  \begin{equation*}
    \Phi(f) := (\tau_J(f|_J))_{J \in \mc{J}}
  \end{equation*}
  is an isometry.
  Since the intervals in $\mc{J}$ are mutually disjont, $\Phi$ is an isometric isomorphism.
  Thus  $\Phi^{-1}$ induces an isometric isomorphism
  \begin{align*}
    \Phi^{-1} \colon c_0\bigl(\mc{J};[V^{q_\bullet}([0,1);Y)]_\theta\bigr) = \bigl[c_0(\mc{J};V^{q_\bullet}([0,1);Y))\bigr]_\theta \to [V_0^{q_\bullet}(\mc{J};Y)]_\theta,
  \end{align*}
  using \cite[Remark 3, \textsection 1.18.1]{hT78}. By \eqref{eq:Vs-single} we have
  \begin{equation*}
    V^{[q_0,q_1]_\theta-\varepsilon}([0,1);Y) \hookrightarrow [V^{q_\bullet}([0,1);Y)]_\theta,
  \end{equation*}
  so that $\Phi^{-1}$ yields an embedding
  \begin{equation*}
    c_0(\mc{J};V^{[q_0,q_1]_\theta-\varepsilon}([0,1);Y)) \hookrightarrow [V_0^{q_\bullet}(\mc{J};Y)]_\theta.
  \end{equation*}
Precomposing with $\Phi$ gives the bounded inclusion
\begin{equation*}
  V_0^{[q_0,q_1]_\theta-\varepsilon}(\mc{J};Y) \hookrightarrow [V_0^{q_\bullet}(\mc{J};Y)]_\theta
\end{equation*}
and completes the proof.
\end{proof}

\section{Fourier multipliers}\label{sec:multiplier}

The Fourier transform and operator-valued Fourier multipliers on vector-valued functions are defined similarly to the scalar-valued case.
Here we just mention that our normalisation of the Fourier transform is
\begin{equation*}
  \wh{f}(\xi) = \FF f(\xi) := \int_{\RR^d} f(t) e^{-2\pi i t\cdot\xi}\dd t, \quad f\in L^1(\RR^d;X), \, \xi \in \RR^d,
\end{equation*}
and that since $\mc{S}(\RR^d) \otimes X$ is dense in $L^p(w;X)$ for every $p \in (1,\infty)$ and $w\in A_\infty$ (see \cite[Ex.\ 9.4.1]{lG09} for the scalar case), the $L^p(w;X) \to L^p(w;Y)$-boundedness of a Fourier multiplier $T_m \colon \mc{S}(\RR^d) \otimes X \to \mc{S}^\prime(\RR^d;Y)$ reduces to the estimate
\begin{equation*}
  \|T_mf\|_{L^p(w;Y)} \lesssim \|f\|_{L^p(w;X)}, \qquad f \in \Sch(\RR^d) \otimes X.
\end{equation*}
Our goal is to find conditions on Banach function spaces $X$ and $Y$ which imply this estimate for $m\in V^s(\Delta;\calL(X,Y))$ and $w$ in a suitable Muckenhoupt class.
We will only consider multipliers $m$ defined on $\RR$; extensions to multipliers defined on $\RR^d$ can be obtained by an induction argument as in \cite[Section 4]{sK14}, \cite{mL07} and \cite{Xu96}, and extensions to multipliers on the torus $\TT$ can be obtained by transference, see \cite[Proposition 4.1]{ALV1}.
In this case one must consider multipliers defined on $\hat{\TT} = \ZZ$, where bounded $s$-variation for a function on $\ZZ$ is defined analogously to Definition \ref{def:svariation}.

We start with a result that is well-known in the unweighted setting (see \cite{HHN02, StrWeis08}).
This is not so important to our main results; it will only be used in the proof of Theorem \ref{thmintermediateUMD}.
Recall that $\Delta = \{ \pm[2^k, 2^{k+1}), k \in \ZZ\}$ is the standard dyadic partition of $\RR$.

\begin{thm}[Vector-valued Marcinkiewicz multiplier theorem]\label{thm:main-multgeneralUMD}
  Let $X$ and $Y$ be UMD Banach spaces, and suppose $\mc{T} \subset \mc{L}_b(X,Y)$ is absolutely convex and $\mc{R}$-bounded.
  Suppose $m \in V^1(\Delta;\mc{T})$.
  Then for all $p \in (1, \infty)$ and $w \in A_{p}$,
  \begin{equation*}
    \nrm{T_m}_{\mc{L}(L^p(w;X), L^p(w;Y))} \leq  \inc_{X,Y,p}([w]_{A_{p}}) \elRr{\mc{T}} \nrm{m}_{V^1(\Delta;\mc{T})}.
  \end{equation*}
\end{thm}
\begin{proof}
  To prove the result one can repeat the argument in \cite[Theorem 4.3]{HHN02} using weighted Littlewood--Paley inequalities with sharp cut-off functions, which can be found for instance in \cite{FHL17} (see also \cite{nL14}).
\end{proof}

Our starting point for multiplier theorems for $m\in V^s$ with $s>1$ is an estimate of Littlewood--Paley--Rubio de Francia type.
For an interval $I\subset \RR$ let
$S_I$ denote the Fourier projection onto $I$, defined by $S_I f := (\mathbf{1}_I \hat{f} )^\vee$ for Schwartz functions $f\in \Sch(\RR)\otimes X$. The following result was obtained in \cite[Theorem 6.5]{ALV1}. Related results have been obtained in \cite{sK14,PSX12}.

\begin{thm}\label{thm:LPRmod-lat}
  Suppose $q\in [2, \infty)$ and let $X$ be a Banach function space such that $X^{q'}\in \UMD$.
  Let $\mc{I}$ be a collection of mutually disjoint intervals in $\RR$.
  Then for all $p > q'$, all $w \in A_{p/q'}$, and all $f\in L^p(w;X)$,
\begin{equation*}
\nrms{\Bigl(\sum_{J\in \Delta} \Bigl( \sum_{\substack{I \in \mc{I}\\I \subset J}} |S_I f|^q \Bigr)^{2/q} \Bigr)^{1/2}}_{L^p(w;X)} \leq \inc_{X,p,q}([w]_{A_{p/q'}}) \nrm{f}_{L^p(w;X)}.
\end{equation*}
\end{thm}

For Hilbert spaces the following variant holds (see \cite[Proposition 6.6 and Remark 6.7]{ALV1}).

\begin{prop}\label{prop:LPRmod-latHilbert}
  Suppose $q\in [2, \infty)$ and let $X$ be a Hilbert space.  Let $\mc{I}$ be a collection of mutually disjoint intervals in $\RR$.
  Then for all $p > q'$, all $w \in A_{p/q'}$ and all $f\in L^p(w;X)$,
\begin{equation*}
\nrms{\Bigl(\sum_{J\in \Delta} \Bigl( \sum_{\substack{I \in \mc{I}\\I \subset J}} \nrm{S_I f}_X^q \Bigr)^{2/q} \Bigr)^{1/2}}_{L^p(w)} \leq \inc_{p,q}([w]_{A_{p/q'}}) \nrm{f}_{L^p(w;X)}.
\end{equation*}
\end{prop}

\subsection{Multipliers in Hilbert spaces}\label{subs:multi-hilbert}
The first part of the following theorem is an analogue of \cite[Theorem A(i)]{sK14}, and the second part is an unweighted analogue of \cite[Theorem A(ii)]{sK14}.
The second part is also proved in \cite[Proposition 3.3]{HP06}.
The exponents $(p,s)$ for which each part of the theorem applies are pictured in Figure \ref{fig:exponents3}.

\begin{thm}\label{thm:HScase}
  Let $X$ and $Y$ be Hilbert spaces, $p,s\in (1, \infty)$, and consider a multiplier $m\in V^s(\Delta;\mc{L}_b(X,Y))$.
  \begin{enumerate}[(i)]
  \item \label{it:HScase1} If $s \leq 2$ and $p \geq s$, then for all $w \in A_{p/s}$ we have
  \begin{align*}
    \| T_m\|_{\calL(L^p(w;X), L^p(w;Y))} &\leq \inc_{p,s}([w]_{A_{p/s}}) \|m\|_{V^s(\Delta;\mc{L}_b(X,Y))}.
  \end{align*}
  \item \label{it:HScase2}If $\frac{1}{s}>\bigl|\frac{1}{p}-\frac{1}{2}\bigr|$ we have
      \begin{align*}
    \| T_m\|_{\calL(L^p(\RR;X), L^p(\RR;Y))} &\lesssim_{p,s} \|m\|_{V^s(\Delta;\mc{L}_b(X,Y))}.
  \end{align*}
  \end{enumerate}
\end{thm}

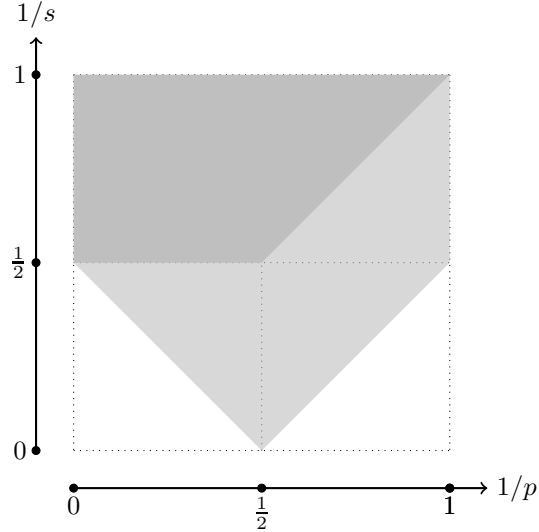
\begin{figure}[h!]
\caption{Allowable exponents for Theorem \ref{thm:HScase}: the weighted case (i) dark shaded, the unweighted case (ii) light shaded.}
\begin{center}
\begin{tikzpicture}[scale=5]
        \pgfmathsetmacro{\THETA}{0.6} 
        \pgfmathsetmacro{\parQ}{(1)} 

	\coordinate (P1) at (0,1/2);
	\coordinate (P2) at (1/2,0);
	\coordinate (P3) at (1,1/2);
	\coordinate (P4) at (1/2,1/2);
    \coordinate (P5) at ({\parQ*(1 - \THETA) + \THETA},1);

    \coordinate (P22) at ({\parQ*(1 - \THETA) + \THETA/2},{\parQ*(1 - \THETA) + \THETA/2});
    \coordinate (P32) at ({\parQ*(1 - \THETA) + \THETA},{\parQ*(1 - \THETA) + \THETA});
    \coordinate (P42) at (0,{\parQ*(1 - \THETA) + \THETA});

	\coordinate (TL) at (0,1);
	\coordinate (TR) at (1,1);

	\draw [thick,->] (0,-0.1) -- (1.1,-0.1);
        \node [right] at (1.1,-0.1) {$1/p$};

	\draw [fill=black] (0,-0.1) circle [radius = .3pt];
	\node [below] at (0,-0.1) {$0$};

	\draw [fill=black] (1,-0.1) circle [radius = .3pt];
	\node [below] at (1,-0.1) {$1$};

        \draw [fill=black] (1/2, -0.1) circle [radius = .3pt];
        \node [below] at (1/2,-0.1) {$\frac{1}{2}$};

        \draw [fill=black] ({\parQ*(1 - \THETA) + \THETA},-0.1) circle [radius = .3pt];
        \node [below] at ({\parQ*(1 - \THETA) + \THETA},-0.1) {$1$};

	\draw [thick,->] (-0.1,0) -- (-0.1,1.1);
	\node [above] at (-0.1,1.1) {$1/s$};

        \draw [fill=black] (-0.1,0) circle [radius = .3pt];
	\node [left] at (-0.1,0) {$0$};

        \draw [fill=black] (-0.1,{1/2}) circle [radius = .3pt];
        \node [left] at (-0.1,{1/2}) {$\frac{1}{2}$};

	\draw [fill=black] (-0.1,1) circle [radius = .3pt];
	\node [left] at (-0.1,1) {$1$};


        \draw [thin,dotted] (0,1) -- (1,1); 
	\draw [thin,dotted] (0,0) -- (1,0); 

        \draw [thin,dotted] (0,{1/2}) -- (1,{1/2});


	\draw [thin,dotted] (0,1) -- (0,0); 
	\draw [thin,dotted] (1,1) -- (1,0); 

        \draw [thin,dotted] (1/2,0) -- (1/2,1);

	\path [fill=lightgray, opacity = 0.6] (TL)--(P1)--(P2)--(P3)--(TR);

        \path [fill=lightgray, opacity = 1] (TL)--(P1)--(P4)--(TR);	
\end{tikzpicture}
\end{center}
\label{fig:exponents3}
\end{figure}

To prove Theorem \ref{thm:HScase} we use the following proposition, which is a version of the first part for $R$-class multipliers.
 The techniques used to prove this proposition are strongly related to those used in the proof of our main result for $\UMD$ Banach function spaces, Theorem \ref{thm:mult-s-var}.

\begin{prop}\label{prop:HScaseProp}
    Let $X$ and $Y$ be Hilbert spaces, $s\in (1, 2]$, and consider a multiplier $m\in R^s(\Delta;\mc{L}_b(X,Y))$.
     Then for all $p > s$ and $w \in A_{p/s}$ we have
  \begin{align*}
    \| T_m\|_{\calL(L^p(w;X), L^p(w;Y))} &\leq \inc_{p,s}([w]_{A_{p/s}}) \|m\|_{R^s(\Delta;\mc{L}_b(X,Y))}.
  \end{align*}
\end{prop}

\begin{proof}
  We only consider the case $s<2$.
  The case $s=2$ is similar, but simpler.
  Fix $\varepsilon>0$ and let $f \in L^p(w;X)$.
  By approximation we may assume that the dyadic Littlewood--Paley decomposition of $f$ has finitely many nonzero terms and set $\Delta_f = \{J \in \Delta: S_Jf \neq 0\}$.
  For each $J \in \Delta_f$ let
  \begin{equation*}
    m|_J = \sum_{k=1}^N \lambda_k a_k^J, \qquad a_k^J = \sum_{I \in \mc{J}_k^J} c_I^{J,k} \mb{1}_I
  \end{equation*}
  be an $R^s(J;\mc{L}_b(X,Y))$-atomic decomposition of the restriction $m|_J$ with $\lambda_k$ independent of $J$ and
  \begin{equation*}
    \sum_{k=1}^N \abs{\lambda_k} \leq (1+\varepsilon)\nrm{m}_{R^\sigma(\Delta;\mc{L}_b(X,Y))}
  \end{equation*}
  as in \cite[Theorem 2.3]{HP06}.
	
  Note that $S_J T_m = T_m S_J$, where we abuse notation by letting $S_J$ denote either the $X$- or $Y$-valued Fourier projection.
  By the Littlewood--Paley estimate (see \cite[Proposition 3.2]{MV15}), H\"older's inequality, Proposition \ref{prop:LPRmod-latHilbert}, and $w \in A_{p/s} \subset A_p$, we have
  \begin{align*}
    \nrm{T_m f}_{L^p(w;Y)} &\leq \inc_p([w]_{A_p}) \nrms{ \has{ \sum_{J \in \Delta_f} \nrm{T_m S_J f}_Y^2 }^{1/2} }_{L^p(w)}\\
                           &\leq \inc_p([w]_{A_p}) \nrms{ \has{\sum_{J \in \Delta_f} \has{ \sum_{k=1}^N \abs{\lambda_k}\sum_{I \in \mc{J}_k^J} \nrm{c_I^{J,k} S_I f}_Y }^2 }^{1/2}}_{L^p(w)}\\
                           &\leq \inc_p([w]_{A_p})   \sum_{k=1}^N |\lambda_k| \,  \nrms{ \has{\sum_{J \in \Delta_f} \bigl( \sum_{I \in \mc{J}^J_k} \nrm{c_I^{J,k}}^s \bigr)^{\frac{2}{s}}\bigl(\sum_{I \in \mc{J}_k^J} \nrm{S_I f}_X^{s^\prime}\bigr)^{\frac{2}{s'}}  }^{\frac12}}_{L^p(w)}  \\
                           & \leq \inc_p([w]_{A_p})  \sum_{k=1}^N |\lambda_k| \, \nrms{ \has{ \sum_{J \in \Delta_f} \has{\sum_{I \in \mc{J}_k^J} \nrm{S_I f}_X^{s^\prime}}^{2/s'} }^{1/2}}_{L^p(w)}\\
                           &\leq \inc_{p,s}([w]_{A_{p/s}})  \sum_{k=1}^N |\lambda_k|  \, \nrm{f}_{L^p(w;X)} .
  \end{align*}
  Since $\varepsilon >0$ was arbitrary this implies
  \begin{equation*}
    \nrm{T_m f}_{L^p(w;Y)} \leq \inc_{p,s}([w]_{A_{p/s}}) \nrm{m}_{R^s(\Delta;\mc{L}_b(X,Y))} \nrm{f}_{L^p(w;X)}
  \end{equation*}
  for all $w \in A_{p/s}$ and $f \in L^p(w;X)$.
\end{proof}

\begin{proof}[Proof of Theorem \ref{thm:HScase}]
  \textbf{Part (i):} We first consider the case $s<p$ and $s<2$.
  Let $w \in A_{p/s}$ and take  $\sigma \in (s,2]$ such that $w \in A_{p/\sigma}$, which is possible by Proposition \ref{prop:muckenhoupt}\eqref{it:mw5}.
  By Lemma \ref{lem:embeddingVHR} we know that $m \in R^\sigma(\Delta;\mc{L}_b(X,Y))$ with
  \begin{equation*}
    \nrm{m}_{R^\sigma(\Delta;\mc{L}_b(X,Y))} \lesssim_{s,\sigma} \nrm{m}_{V^s(\Delta;\mc{L}_b(X,Y))},
  \end{equation*}
  so by Proposition \ref{prop:HScaseProp} we obtain
  \begin{equation*}
    \nrm{T_m}_{\mc{L}(L^p(w;X),L^p(w;Y))} \leq \inc_{p,s}([w]_{A_{p/s}}) \nrm{m}_{V^s(\Delta;\mc{L}_b(X,Y))}.
  \end{equation*}

  Next we consider the case $p>s=2$.
  Observe that by \cite[Proposition 5.3.16]{HNVW16} it suffices to prove the result for the truncated multipliers $$m_N := \mb{1}_{\bigcup_{n = 1}^N J_n} m,$$
  where $\Delta = (J_n)_{n = 1}^\infty$ is an arbitrary ordering of $\Delta$.
  Since $m_N \in V_0^s(\Delta;\mc{L}_b(X,Y))$ uniformly, without loss of generality we may work with an arbitrary decaying multiplier $m \in V_0^s(\Delta;\mc{L}_b(X,Y))$.
  Fix $w \in A_{p/2}$.
  Then by Proposition \ref{prop:muckenhoupt}\eqref{it:mw2} there exists a $\delta >0$ such that $w^{1+\delta} \in A_{p/2}$.
  Take
  \begin{equation*}
    \theta = \frac{2}{p}\has{1-\frac{1}{1+\delta}}, \qquad p_0 = (1+\delta)(1-\theta)p, \quad \text{and} \quad \sigma = 2-\theta.
  \end{equation*}
  Then $\theta \in (0,1)$, $\sigma \in (1,2)$  and $p_0 = p+ (p-2)\delta > p$, so by the first case we have
  \begin{equation*}
    \nrm{T_m}_{\mc{L}(L^{p_0}(w;X),L^{p_0}(w;Y))} \leq \inc_{p_0,\sigma}([w]_{A_{p/2}}) \nrm{m}_{V_0^\sigma(\Delta;\mc{L}_b(X,Y))}.
  \end{equation*}
  Moreover by Plancherel's theorem (which is valid since $X$ and $Y$ are Hilbert spaces) we know that
  \begin{equation}\label{eq:plancherelappl}
    \| T_m\|_{\calL(L^2(\RR;X), L^2(\RR;Y))} \leq \|m\|_{L^\infty(\RR;\mc{L}_b(X,Y))}.
  \end{equation}
  Since
  \begin{equation*}
    \frac{1}{[p_0,2]}_\theta = \frac{1}{p(1+\delta)} + \frac{1}{p} - \frac{1}{p(1+\delta)} =\frac{1}{p},
  \end{equation*}
  we know by \cite[Theorem 1.18.5]{hT78} that $L^{p}(w;X) = [L^{p_0}(w^{1+\delta},X), L^{2}(\RR;X)]_{\theta}$, and likewise with $X$ replaced by $Y$.
  Moreover since $[\sigma,\infty]_\theta=\frac{2-\theta}{1-\theta}>2$ we have the continuous inclusions
  \begin{align*}
    V^2(\Delta;\mc{L}_b(X,Y)) &\hookrightarrow [V_0^\sigma(\Delta;\mc{L}_b(X,Y)), V^\infty_0(\Delta;\mc{L}_b(X,Y))]_\theta \\&\hookrightarrow [V_0^\sigma(\Delta;\mc{L}_b(X,Y)), L^\infty(\RR;\mc{L}_b(X,Y))]_\theta
  \end{align*}
  by Theorem \ref{thm:VR-interpoln}.
  By bilinear complex interpolation \cite[\textsection 4.4]{BL76} applied to the bilinear map $(m,f) \mapsto T_mf$ we have boundedness of $T_m \colon L^{p}(w;X) \to L^{p}(w;Y)$ with the required norm estimate.

  Finally we consider the case $p=s \geq 2$; we will use another interpolation argument.
  Fix $w\in A_1$.
  Then by Proposition \ref{prop:muckenhoupt}\eqref{it:mw2} there exists a $\delta>0$ such that $w^{1+\delta}\in A_1$.
  Fix $p_1\in (s, s+(s-1)\delta)$.
  By the argument of the previous cases we have
  \begin{equation*}
    \|T_m\|_{\calL(L^{p_1}(w^{1+\delta};X), L^{p_1}(w^{1+\delta};Y))} \leq \inc_{p_1,s}([w]_{A_1}) \nrm{m}_{V^s(\Delta;\mc{L}_b(X,Y))}.
  \end{equation*}
  Let $\theta\in (0,1)$ be such that $\theta(1+\delta) s = p_1$; such a $\theta$ exists since $p_1<s+(s-1)\delta$.
  Choose $p_0\in (1, s)$ such that $[p_0,p_1]_{\theta} = s$.
  Such a $p_0$ exists since $p_1>s$ and $[1,p_1]_{\theta}<s$.
  Indeed, the latter follows from
  \begin{equation*}
    \frac{s}{[1,p_1]_{\theta}} = s(1-\theta)+s\frac{\theta}{p_1} =  s - \frac{p_1}{1+\delta} + \frac{1}{1+\delta}  > 1.
  \end{equation*}
  Since $p_0<s \leq 2$ we have by duality with the previous cases (taking $w=1$) that
  \[\|T_m\|_{\calL(L^{p_0}(\RR;X),L^{p_0}(\RR;Y))}\lesssim_{p_0,s}\nrm{m}_{V^s(\Delta;\mc{L}_b(X,Y))}.\]
  As before our choice of $\theta$ yields $L^{s}(w;X) = [L^{p_0}(\RR,X), L^{p_1}(w^{1+\delta};X)]_{\theta}$, and likewise with $X$ replaced by $Y$.
  Therefore by complex interpolation we have boundedness of $T_m \colon L^{s}(w;X) \to L^{s}(w;Y)$ with the required norm estimate.

\bigskip

\textbf{Part (ii):}
The case $p=2$ is clear from \eqref{eq:plancherelappl} and the embedding of the $V^s$ classes in $L^\infty$.
For $p>2$ we may assume without loss of generality that $m \in V_0^s(\Delta;\mc{L}_b(X,Y))$ as in part (i).
Moreover, by embedding of the $V^s$ classes, we may assume that $s >2$.

Let $\sigma \in \bigl(s, \bigl(\frac12-\frac1p\bigr)^{-1}\bigr)$ and fix $t \in (2,\infty)$ such that $[2,t]_{\frac{\sigma}{2}} = p$. Such a $t$ exists since $p>2$ and
  \begin{equation*}
    \frac{1}{p} = \frac{1}{[2,t]_{\frac{\sigma}{2}}} = \frac{1}{2} - \frac{1}{\sigma} + \frac{2}{\sigma} \frac{1}{t},
  \end{equation*}
  which implies that
  \begin{equation*}
  \frac{1}{t} = \frac{2}{s}\has{\frac{1}{p}+\frac{1}{\sigma}-\frac12}>0.
  \end{equation*}
  Using the boundedness properties
  \begin{align*}
    V_0^\infty(\Delta;\mc{L}_b(X,Y)) &\times L^2(\RR;X)\to L^2(\RR;Y) \quad \text{and} \\
    V_0^2(\Delta;\mc{L}_b(X,Y)) &\times L^t(\RR;X)\to L^t(\RR;Y)
  \end{align*}
  of the bilinear map $(m,f) \mapsto T_mf$, which follow from \eqref{eq:plancherelappl} and part (i) respectively, we have boundedness of $T_m \colon L^{p}(w;X) \to L^{p}(w;Y)$ with the required norm estimate by bilinear complex interpolation  \cite[\textsection 4.4]{BL76}.
  Here we use \cite[Theorem 1.18.4]{hT78} and Theorem \ref{thm:VR-interpoln} to identify the interpolation spaces as before.
  The case $p<2$ follows by a duality argument.
\end{proof}

\begin{rmk}\
\begin{enumerate}
\item If the multiplier is scalar-valued and $X = Y$, then Theorem \ref{thm:HScase} follows simply from the scalar case and a standard Hilbert space tensor extension argument (see \cite[Theorem 2.1.9]{HNVW16}).
\item As in \cite[Theorem A]{sK14}, a weighted version of Theorem \ref{thm:HScase}\eqref{it:HScase2} can be proved, but we omit it to prevent things from getting too complicated.
\end{enumerate}
\end{rmk}

\subsection{Multipliers in UMD Banach function spaces} \label{subs:multi-banach}
We now turn to our main result (Theorem \ref{thm:mult-s-var}).
Its proof is inspired by that of \cite[Theorem 2.3]{HP06}, which is a generalisation of the Hilbert space result in Theorem \ref{thm:HScase}.
Besides the regularity assumption on the multiplier as in the Hilbert space case, we will need an $\ell^2(\ell^q)$-boundedness assumption. We first prove a result for $R$-class multipliers, analogous to Proposition \ref{prop:HScaseProp}.

\begin{prop}\label{prop:main-mult}
  Let $q \in (1,2]$, $p \in (q, \infty)$, and $w \in A_{p/q}$.
  Let $X$ and $Y$ be Banach function spaces with $X^q \in \UMD$ and $Y \in \UMD$.
  Let $\mc{T} \subset \mc{L}_b(X,Y)$ be absolutely convex and $\el{2}{q^\prime}$-bounded, and suppose $m \in R^q(\Delta;\mc{T})$.
  Then
  \begin{equation*}
    \nrm{T_m}_{\mc{L}(L^p(w;X), L^p(w;Y))} \leq  \inc_{X,Y,p,q}([w]_{A_{p/q}}) \elR{2}{q^\prime}{\mc{T}} \nrm{m}_{R^q(\Delta;\mc{T})}.
  \end{equation*}
\end{prop}

\begin{proof}
  Fix $\varepsilon>0$ and let $f \in L^p(w;X)$.
  We begin as in the proof of Proposition \ref{prop:HScaseProp}, which began as in the proof of \cite[Theorem 2.3]{HP06}: we assume that the dyadic Littlewood--Paley decomposition of $f$ has finitely many nonzero terms and set $\Delta_f = \{J \in \Delta: S_Jf \neq 0\}$.
  For each $J \in \Delta_f$ let
  \begin{equation*}
    m|_J = \sum_{k=1}^N \lambda_k a_k^J, \qquad a_k^J = \sum_{I \in \mc{J}_k^J} c_I^{J,k} \mb{1}_I
  \end{equation*}
  be a $R^q(J;\mc{T})$-atomic decomposition of the restriction $m|_J$ with $\lambda_k$ independent of $J$, with each $\mc{J}_k^J$ finite, and with
  \begin{equation*}
    \sum_{k=1}^N \abs{\lambda_k} \leq (1+\varepsilon)\nrm{m}_{R^q(\Delta;\mc{L}_b(X,Y))}.
  \end{equation*}

  As before, $S_J T_m = T_m S_J$.
  By the Littlewood--Paley theorem for UMD Banach function spaces (see \cite[Proposition 6.1]{ALV1}), using that $Y \in \UMD$ and $w \in A_{p/q} \subset A_p$, we have
  \begin{align*}
    \nrm{T_m f}_{L^p(w;Y)} &\leq \inc_{Y,p}([w]_{A_p}) \nrms{ \has{ \sum_{J \in \Delta_f} |T_m S_J f|^2 }^{1/2} }_{L^p(w;Y)}\\
                           &= \inc_{Y,p}([w]_{A_p}) \nrms{ \has{ \sum_{J \in \Delta_f} \abss{ \sum_{k=1}^N \lambda_k \sum_{I \in \mc{J}_k^J} c_I^{J,k} S_I f }^2 }^{1/2} }_{L^p(w;Y)} \\
                           &\leq \inc_{Y,p}([w]_{A_p}) \sum_{k=1}^N |\lambda_k| \nrms{ \has{ \sum_{J \in \Delta_f} \abss{ \sum_{I \in \mc{J}_k^J} c_I^{J,k} S_I f }^2 }^{1/2} }_{L^p(w;Y)}.
  \end{align*}
  We estimate the sum on the right hand side by
  \begin{align*}
    &\sum_{k=1}^N |\lambda_k| \nrms{ \has{ \sum_{J \in \Delta_f} \abss{ \sum_{I \in \mc{J}_k^J} c_I^{J,k} S_I f }^2 }^{1/2} }_{L^p(w;Y)} \\
    &\leq \sum_{k=1}^N |\lambda_k| \biggl\| \biggl( \sum_{J \in \Delta_f} \has{ \bigl(\sum_{I \in \mc{J}_k^J} \nrm{c_I^{J,k}}_{\mc{T}}^q \bigr)^{1/q} \bigl( \sum_{I \in \mc{J}_k^J} \abss{\frac{c_I^{J,k}S_I f}{\nrm{c_I^{J,k}}_{\mc{T}}} }^{q^\prime} \bigr)^{1/q^\prime} }^2 \biggr)^{1/2} \biggr\|_{L^p(w;Y)} \\
    &\leq \sum_{k=1}^N |\lambda_k| \biggl\| \biggl( \sum_{J \in \Delta_f} \has{ \sum_{I \in \mc{J}_k^J} \abss{\frac{c_I^{J,k}}{\nrm{c_I^{J,k}}_{\mc{T}}} S_I f}^{q^\prime} }^{2/q^\prime}  \biggr)^{1/2} \biggr\|_{L^p(w;Y)}.
  \end{align*}
  By the definition of the Minkowski norm, the operators $c_I^{J,k}/\nrm{c_I^{J,k}}_\mc{T}$ all lie in $\mc{T}$, so by $\el{2}{q^\prime}$-boundedness of $\mc{T}$ we have
  \begin{align*}
    &\nrm{T_m f}_{L^p(w;Y)} \\
    &\leq \inc_{Y,p}([w]_{A_p}) \elR{2}{q^\prime}{\mc{T}} \sum_{k=1}^N  |\lambda_k| \nrms{ \has{ \sum_{J \in \Delta_f} \big( \sum_{I \in \mc{J}_k^J}  |S_I f|^{q^\prime} \big)^{2/q^\prime} }^{1/2} }_{L^p(w;X)}.
  \end{align*}
  By Theorem \ref{thm:LPRmod-lat},
  \begin{align*}
    \nrms{ \has{ \sum_{J \in \Delta_f} \big( \sum_{I \in \mc{J}_k^J}  |S_I f|^{q^\prime} \big)^{2/q^\prime} }^{1/2} }_{L^p(w;X)}
    &\leq \inc_{X,p,q}([w]_{A_{p/q}}) \nrm{f}_{L^p(w;X)}.
  \end{align*}
  Since $\sum_{k=1}^N  |\lambda_k| \leq (1+\varepsilon) \nrm{m}_{R^q(\Delta;\mc{T})}$ and $\varepsilon > 0$ was arbitrary, this finishes the proof.
\end{proof}

Our main multiplier theorem follows easily.
Recall that $w \in \alpha_{p,q}$ if and only if $w^{1-p'} \in A_{p'/q'}$ with $[w]_{\alpha_{p,q}} := [w^{1-p^\prime}]_{A_{p^\prime/q^\prime}}$.

\begin{thm}\label{thm:mult-s-var}
  Let $X$ and $Y$ be Banach function spaces, and let $\mc{T} \subset \mc{L}_b(X,Y)$ be absolutely convex. Let $q \in (1,2]$, $s \in [1,q)$ and $m \in V^s(\Delta;\mc{T})$.
  \begin{enumerate}[(i)]
  \item \label{it:mult-s1}
    Suppose that $X^q \in \UMD$, $Y \in \UMD$, and $\mc{T}$ is $\el{2}{q^\prime}$-bounded. Then for all $p \in (q,\infty)$ and $w \in A_{p/q}$ we have
    \begin{equation*}
      \nrm{T_m}_{\mc{L}(L^p(w;X), L^p(w;Y))} \leq \inc_{X,Y,p,q}([w]_{A_{p/q}}) \elR{2}{q^\prime}{\mc{T}} \nrm{m}_{V^s(\Delta;\mc{T})}.
    \end{equation*}
  \item \label{it:mult-s2}
    Suppose that $X \in \UMD$, $(Y^*)^{q} \in \UMD$, $\mc{T}$ is $\el{2}{q}$-bounded, and $m \in V^{s}(\Delta;\mc{T})$.
    Then for all  $p \in (1, q')$ and  $w \in \alpha_{p,q'}$ we have
    \begin{equation*}
      \nrm{T_m}_{\mc{L}(L^p(w;X), L^p(w;Y))} \leq \inc_{X,Y,p,q}([w]_{\alpha_{p,q}}) \elR{2}{q}{\mc{T}} \nrm{m}_{V^{s}(\Delta;\mc{T})}.
    \end{equation*}
  \end{enumerate}
\end{thm}

\begin{proof}
The first part follows directly from Proposition \ref{prop:main-mult} and Lemma \ref{lem:embeddingVHR}. For the second part a standard duality argument shows that
\begin{equation*}
  \nrm{T_m}_{\mc{L}(L^p(w;X),L^p(w;Y))} \leq \nrm{T_{m^*}}_{\mc{L}(L^{p^\prime}(w^{1-p^\prime};Y^*), L^{p^\prime}(w^{1-p^\prime};X^*))},
\end{equation*}
with $\map{m^*}{\RR}{\spn (\mc{T}^*)}$ defined by $m^*(t) = m(t)^*$ for all $t \in \RR$. Applying the first part to $m^*$, using Proposition \ref{prop:lrsdual} to show that $\mc{T^*}$ is $\el{2}{q'}$-bounded and noting that $m^* \in V^{q}(\Delta;\mc{T}^*)$, completes the proof.
\end{proof}

If $q=2$ and $w=1$ in Theorem \ref{thm:mult-s-var}, we recover \cite[Corollary 2.5]{HP06} for Banach function spaces, except for the endpoint $p=2$, which is missing since we worked in the weighted setting. If the multiplier is scalar-valued and $X = Y$, the result was proved in \cite{ALV1} using vector-valued extrapolation.

\begin{rmk}\label{rem:Rnecessity}
The $\el{2}{q^\prime}$-boundedness assumption in Theorem \ref{thm:mult-s-var} arises naturally from the proof.
It is known that boundedness of $T_m$ implies $\mc{R}$-boundedness---and thus $\ell^2$-boundedness if $X$ has finite cotype---of the image of the Lebesgue points of $m$ (see \cite{ClePru01} or \cite[Theorem 5.3.15]{HNVW16}).
However, $\el{2}{q'}$-boundedness is not necessary, as may be seen by considering $m = n S$ where $n\in R^q(\Delta)$ is a scalar multiplier and $S\colon X\to Y$ is a bounded linear operator.
In this case $T_m$ will be bounded, but $\{S\}$ need not be $\el{2}{q'}$-bounded for $q\neq 2$ (see \cite[Example 2.16]{KU14}).
\end{rmk}

Using complex interpolation, the reverse H\"older inequality, and the openness of the $\UMD$ property, we can obtain a result for the endpoint $p=q=s$ in Theorem \ref{thm:mult-s-var}.

\begin{prop}\label{prop:A1-est}
	Let $X$ and $Y$ be Banach function spaces.
	Let $q,r \in (1,2)$ and suppose that $X^q \in \UMD$ and $(Y^*)^r \in \UMD$.
        Let $\mc{T} \subset \mc{L}(X,Y)$ be absolutely convex and both $\el{2}{q'}$- and $\el{2}{r}$-bounded.
        Let $s = \min\{q,r\}$ and suppose that $m \in V^{s}(\Delta;\mc{T})$.
	Then for all $w \in A_1$,
	\begin{equation*}
          \nrm{T_m}_{\mc{L}(L^q(w;X),L^q(w;Y))} \leq \inc_{X,Y,q,r}([w]_{A_1}) \max\{\elR{2}{q'}{\mc{T}},\elR{2}{r}{\mc{T}}\} \nrm{m}_{V^{s}(\Delta;\mc{T})}.
	\end{equation*}
      \end{prop}

      \begin{proof}
        Fix $w \in A_1$, so that by Proposition \ref{prop:muckenhoupt}\eqref{it:mw2} there exists an $\delta>0$ such that $w^{1+\delta} \in A_1$.
        By the openness of the $\UMD$ property we know that there exist $q_0 \in (q,q+(q-1)\delta)$ and $r_0 \in (r,\infty)$ such that $X^{q_0}, (Y^*)^{r_0} \in \UMD$.
        By Corollary \ref{cor:lrstriangle} we know that $\mc{T}$ is $\el{2}{q_0'}$- and $\mc{T}$ is $\el{2}{r_0'}$-bounded with
        \begin{equation}\label{eq:r0q0rq}
          \elR{2}{q_0'}{\mc{T}} \leq \elR{2}{q'}{\mc{T}} \quad \text{and} \quad \elR{2}{r_0'}{\mc{T}} \leq \elR{2}{r'}{\mc{T}}.
        \end{equation}
        Fix $p_1 \in (q_0, q+(q-1)\delta)$.
        By Theorem \ref{thm:mult-s-var}\eqref{it:mult-s1} and \eqref{eq:r0q0rq} we know that
        \begin{align*}
          \nrm{T_m}_{\mc{L}(L^{p_1}(w^{1+\delta};X),L^{p_1}(w^{1+\delta};Y))} &\leq \inc_{X,Y,p_1,q_0}([w]_{A_1}) \elR{2}{q^\prime}{\mc{T}} \nrm{m}_{V^s(\Delta;\mc{T})}.
        \end{align*}
        Let $\theta \in (0,1)$ be such that $\theta(1+\delta)q=p_1$, and fix $p_0 \in (1,q)$ such that $[p_0,p_1]_\theta = q$. These parameters exist by the same argument as in Theorem \ref{thm:HScase}\eqref{it:HScase1}.
        Since $p_0 < r_0'$, we know by Theorem \ref{thm:mult-s-var}\eqref{it:mult-s2} and \eqref{eq:r0q0rq} that
        \begin{align*}
          \nrm{T_m}_{\mc{L}(L^{p_0}(\RR;X),L^{p_0}(\RR;Y))} &\lesssim_{X,Y,p_0,r_0} \elR{2}{r}{\mc{T}}\nrm{m}_{V^s(\Delta;\mc{T})}.
        \end{align*}
        Therefore by complex interpolation as in Theorem \ref{thm:HScase}\eqref{it:HScase1} we have boundedness of $\map{T_m}{L^{q}(w;X)}{L^{q}(w;Y)}$ with the required norm estimate.
      \end{proof}

      When dealing with operator-valued multipliers $m$, to check the hypotheses of our results, one needs an $\el{2}{q^\prime}$-bounded subset $\mc{T} \subset \mc{L}_b(X,Y)$ whose span contains $m(\RR)$, such that $m$ has the appropriate regularity \emph{when measured with respect to the Minkowski norm induced by $\mc{T}$}.
      An obvious na\"ive choice is to assume that $m(\RR)$ is $\el{2}{q^\prime}$-bounded and to take $\mc{T} = \overline{m(\RR)}$, but $m$ may not be sufficiently regular with respect to the $\mc{T}$-Minkowski norm.
      By making $\mc{T}$ larger $m$ becomes more regular in the $\mc{T}$-Minkowski norm, but enlarging $\mc{T}$ may violate $\el{2}{q^\prime}$-boundedness.
      Constructing such a set $\mc{T}$ given a general multiplier $m$ is quite subtle (except of course in the scalar case, where the Minkowski norm on the one-dimensional span of $m$ is equivalent to the absolute value on $\CC$).
      Below we give an example where these problems may be surmounted using extrapolation techniques.

      \begin{prop}\label{prop:holder-lrs}
        Let $\alpha \in (0,1]$.
        Suppose that $\map{m}{\RR}{\mc{L}(\Sigma(\RR^d),L^0(\RR^d))}$ and that for some $p_0\in (1, \infty)$ and all $w\in A_{p_0}$ the following H\"older-type condition is satisfied:
        \begin{equation}\label{eqn:holder-est-m}
          \sup_{x\in \RR}\|m(x)\|_{\mc{L}(L^{p_0}(w))}+ \sup_{J \in \Delta} |J|^{\alpha} [m]_{C^\alpha(J;\mc{L}_b(L^{p_0}(w)))} \leq \inc ([w]_{A_{p_0}}).
        \end{equation}
        Then there exists a subset $\mc{T} \subset \mc{L}(\Sigma(\RR^d),L^0(\RR^d))$ such that $m \in V^{1/\alpha}(\Delta;\mc{T})$ and
        ${\mc{T}}$ is $\el{u}{v}$-bounded on $L^p(w)$ for all $p,u,v \in (1,\infty)$ and $w \in A_p$, with
        \begin{equation*}
          \elR{u}{v}{{\mc{T}}} \leq \inc_{p,u,v} ( [w]_{A_p} ).
        \end{equation*}
      \end{prop}

      \begin{proof}
        For each $J \in \Delta$ define
        \begin{equation*}
          \mc{T}(J) := m(J) \cup \cbraces{\frac{m(x) - m(y)}{|x-y|^\alpha} |J|^{\alpha}: x \neq y \in J },
        \end{equation*}
        and set $\mc{T} := \bigcup_{J \in \Delta} \mc{T}(J)$.
        Note that $m(\RR) \subset \mc{T}$.
        We will show that $\mc{T}$ has the desired properties.

        Since $m(x)\in \mc{T}$ and $\frac{m(x) - m(y)}{|x-y|^\alpha} |J|^{\alpha} \in \mc{T}$ for all $J \in \Delta$ and all $x \neq y \in J$, by the definition of the Minkowski and H\"older norms, we have $\|m(x)\|_{{\mc{T}}}\leq 1$ and $|J|^{\alpha}[m]_{C^\alpha(J;{\mc{T}})} \leq 1$, from which it follows directly that $m \in V^{1/\alpha}(\Delta;\mc{T})$.

        By scalar extrapolation (see \cite[Theorems 3.9 and Corollary 3.14]{CMP11}), we have \eqref{eqn:holder-est-m} for all $p\in (1, \infty)$, which implies that
        \begin{equation*}
          \nrm{Tf}_{L^{p}(w)}  \leq \inc_p ( [w]_{A_{p}} ) \nrm{f}_{L^p(w)}
        \end{equation*}
        for all $p \in (1,\infty)$, $w \in A_p(\RR^d)$, $f \in L^p(w)$, and $T \in \mc{T}$.
        Thus the $\el{u}{v}$-boundedness result follows directly from Proposition \ref{prop:unifbd-ls}.
      \end{proof}

      In the next example we specialise to the case $X = Y = L^r$ and $s \in (1,2)$.
      Results for $s \in [2,\infty)$ will be presented in Example \ref{ex:Lr-large-s}.
      Note that the $\ell^2$-boundedness or $\el{2}{s}$-boundedness assumptions can be deduced for instance from weight-uniform H\"older estimates as in Proposition \ref{prop:holder-lrs}.

      \begin{example}\label{ex:Lr-small-s}
        Let $p,r \in (1,\infty)$ and let $\mc{T} \subset \mc{L}_b(L^r)$ be absolutely convex. Let $s\in (1, 2)$ and $m \in V^s(\Delta;\mc{T})$. Then $T_m$ is bounded on $L^p(w;L^r)$ in each of the following cases:
        \begin{enumerate}[(i)]
        \item If $r=2$,
\begin{enumerate}[(a)]
\item $p \in [s,\infty)$ and $w\in A_{p/s}$.
\item $p\in (1, s']$ and $w\in \alpha_{p,s'}$.
\end{enumerate}
\item If $r \in (2,\infty)$,
\begin{enumerate}[(a)]
  \item $p\in (2, \infty)$, $w\in A_{p/2}$ and $\mc{T}$ is $\ell^{2}$-bounded.
  \item $p \in (1,r)$, $s \in (1,r')$, $w\in \alpha_{p,s'}$ and $\mc{T}$ is $\el{2}{s}$-bounded.
\end{enumerate}
\item If $r \in (1,2)$,
\begin{enumerate}[(a)]
  \item $p\in (1, 2)$, $w \in \alpha_{p,2}$ and $\mc{T}$ is $\ell^{2}$-bounded.
  \item $p \in (r,\infty)$, $s \in (1,r)$, $w\in A_{p/s}$ and $\mc{T}$ is $\el{2}{s'}$-bounded.
\end{enumerate}
\end{enumerate}
\end{example}

\begin{proof}
The case (i)(a) follows from Theorem \ref{thm:HScase} and the case (i)(b) from a duality argument.   The cases (ii)(a) and (iii)(a) follow from Theorem \ref{thm:mult-s-var}\eqref{it:mult-s1} and \eqref{it:mult-s2} with $q=2$. For (iii)(b) choose $q\in (s,r)$ such that $w\in A_{p/q}$. By Corollary \ref{cor:lrstriangle}, $\mc{T}$ is $\el{2}{q'}$-bounded, and therefore Theorem \ref{thm:mult-s-var}\eqref{it:mult-s1} applies. Similarly, (ii)(b) follows from Theorem \ref{thm:mult-s-var}\eqref{it:mult-s2}.
\end{proof}

There is some overlap between the cases (a) and (b) in Example \ref{ex:Lr-small-s}, but the classes of weights considered are difficult to compare.
For $X = L^r$, we can exploit that we always have either $X^2 \in \UMD$ or $(X^*)^2 \in \UMD$.
This is not possible for general $\UMD$ Banach function spaces, which restricts the class of multipliers that can be handled by our results, as shown in the following example.

\begin{example} Let $p \in (1,\infty)$, $r \in (1,2)$, and let $\mc{T} \subset \mc{L}_b(L^r\oplus L^{r'})$ be absolutely convex. Let $s\in (1, r)$ and $m\in V^s(\Delta;\mc{T})$. Then $T_m$ is bounded on $L^p(w;L^r\oplus L^{r'})$ in each of the following cases:
\begin{enumerate}[(i)]
\item $p \in (r,\infty)$, $w \in A_{p/s}$ and $\mc{T}$ is $\el{2}{s'}$-bounded.
\item $p\in (1,r')$, $w\in \alpha_{p,s'}$ and $\mc{T}$ is $\el{2}{s}$-bounded.
\end{enumerate}
The result follows from Theorem \ref{thm:mult-s-var} in the same way as in Example \ref{ex:Lr-small-s}.
\end{example}

\subsection{Multipliers in intermediate UMD Banach function spaces\label{subs:UMDintermediate}}

We can prove stronger results, allowing for multipliers of lower regularity, if we consider `intermediate' spaces $X = [Y,H]_\theta$ where $Y^q \in \UMD$ for some $q \in (1,2]$ and $H$ is a Hilbert space.
For example, when $r \in (2,\infty)$, we have $L^r = [L^{r_0},L^2]_\theta$ for some $r_0 \in (r, \infty)$ and $\theta \in (0,1)$. In this case $Y = L^{r_0}$ satisfies the conditions of Theorem \ref{thm:mult-s-var}\eqref{it:mult-s1} with $q=2$ and with $H = L^2$ we can use Theorem \ref{thm:HScase}.

In order to use interpolation methods we will need that $\text{span}(\mathcal{T})$ with the Minkow\-ski norm is a Banach space, i.e. that  $\mathcal{T}$ is a Banach disc (see below Definition \ref{def:svariation}).

\begin{thm}\label{thm:mult-interp-edit}
  Let $p \in (1,\infty)$, $q \in (1,2]$ and $\theta \in (0,1)$.
  Let $Y$ and $H$ be Banach function spaces over the same measure space, with $Y^q \in \UMD$, $H$ a Hilbert space, and $Y \cap H$ dense in both $Y$ and $H$. Let $X = [Y,H]_\theta$.
  Suppose $\mc{T} \subset \mc{L}_b(Y \cap H)$ is a Banach disc which is $\el{2}{q^\prime}$-bounded on $Y$ and uniformly bounded on $H$.
  Let $s \in (1,\infty)$ and suppose that $m \in V^s(\Delta;\mc{T})$.
  \begin{enumerate}[(i)]
  \item \label{it:interp1} If $s < \min\cbrace{p,[q,2]_\theta}$ and $s \geq [q,1]_\theta$, then
    \begin{equation*}
      \|T_m\|_{\mc{L}(L^p(w;X))}\leq \inc_{Y,p, q, s,  \theta}([w]_{A_{p/s}}) \|m\|_{V^{s}(\Delta;\mc{T})} \elR{2}{q'}{\mc T}
    \end{equation*}
    for all $w \in A_{p/s}$.
  \item \label{it:interp2} If
    \begin{equation*}
      \frac{1}{s} > \max\cbraces{ \frac{1}{[q,2]_\theta}-\frac{1}{p}, \frac{1-\theta}{q}, \frac1p -\frac{\theta}{2}}
    \end{equation*}
    and $p > [q,1]_\theta$, then
    \begin{equation*}
      \|T_m\|_{\mc{L}(L^p(\RR;X))}\lesssim_{Y, p,q, s, \theta}  \|m\|_{V^{s}(\Delta;\mc{T})} \elR{2}{q'}{\mc T}.
    \end{equation*}
  \end{enumerate}
\end{thm}

The allowable exponents $(p,s)$ in Theorem \ref{thm:mult-interp-edit} are shown in Figure \ref{fig:exponents}.
The symmetry in Figure \ref{fig:exponents} is due to the equalities
\begin{equation*}
  \frac{\theta}{2} = \frac{1}{[\infty,2]_\theta} - 0 = \frac{1}{[q,1]_\theta} - \frac{1}{[q,2]_\theta}  = \frac{1}{[q,2]_\theta} - \frac{1}{[q,\infty]_\theta}
\end{equation*}
and
\begin{equation*}
  \frac{1-\theta}{q} = \frac{1}{[q,\infty]_\theta} -0 = \frac{1}{[q,2]_\theta} - \frac{1}{[\infty,2]_\theta}.
\end{equation*}

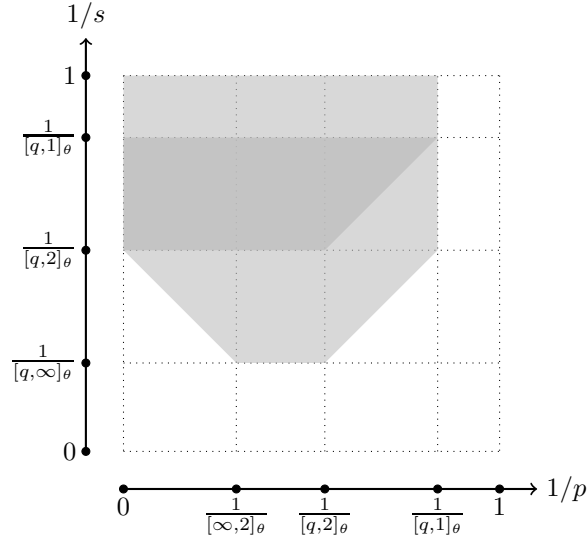
\begin{figure}[h!!]
\caption{Allowable exponents for Theorem \ref{thm:mult-interp-edit}: the weighted case (i) dark shaded, the unweighted case (ii) light shaded.}
\begin{center}
\begin{tikzpicture}[scale=5]

        \pgfmathsetmacro{\THETA}{0.6} 
        \pgfmathsetmacro{\parQ}{(1/1.7)} 

	\coordinate (P1) at (0,{\parQ*(1 - \THETA) + \THETA/2});
	\coordinate (P2) at (\THETA/2,{\parQ*(1-\THETA)});
	\coordinate (P3) at ({\parQ*(1 - \THETA) + \THETA/2},{\parQ*(1-\THETA)});
	\coordinate (P4) at ({\parQ*(1 - \THETA) + \THETA},{\parQ*(1 - \THETA) + \THETA/2});
    \coordinate (P5) at ({\parQ*(1 - \THETA) + \THETA},1);

    \coordinate (P22) at ({\parQ*(1 - \THETA) + \THETA/2},{\parQ*(1 - \THETA) + \THETA/2});
    \coordinate (P32) at ({\parQ*(1 - \THETA) + \THETA},{\parQ*(1 - \THETA) + \THETA});
    \coordinate (P42) at (0,{\parQ*(1 - \THETA) + \THETA});

	\coordinate (TL) at (0,1);
	\coordinate (TR) at (1,1);

	\draw [thick,->] (0,-0.1) -- (1.1,-0.1);
        \node [right] at (1.1,-0.1) {$1/p$};

	\draw [fill=black] (0,-0.1) circle [radius = .3pt];
	\node [below] at (0,-0.1) {$0$};


	\draw [fill=black] (1,-0.1) circle [radius = .3pt];
	\node [below] at (1,-0.1) {$1$};

        \draw [fill=black] (\THETA/2, -0.1) circle [radius = .3pt];
        \node [below] at (\THETA/2,-0.1) {$\frac{1}{[\infty,2]_\theta}$};

        \draw [fill=black] ({\parQ*(1 - \THETA) + \THETA/2},-0.1) circle [radius = .3pt];
        \node [below] at ({\parQ*(1 - \THETA) + \THETA/2},-0.1) {$\frac{1}{[q,2]_\theta}$};

        \draw [fill=black] ({\parQ*(1 - \THETA) + \THETA},-0.1) circle [radius = .3pt];
        \node [below] at ({\parQ*(1 - \THETA) + \THETA},-0.1) {$\frac{1}{[q,1]_\theta}$};

	\draw [thick,->] (-0.1,0) -- (-0.1,1.1);
	\node [above] at (-0.1,1.1) {$1/s$};

        \draw [fill=black] (-0.1,0) circle [radius = .3pt];
	\node [left] at (-0.1,0) {$0$};

        \draw [fill=black] (-0.1,{\parQ*(1 - \THETA)}) circle [radius = .3pt];
        \node [left] at (-0.1,{\parQ*(1 - \THETA)}) {$\frac{1}{[q,\infty]_\theta}$};

        \draw [fill=black] (-0.1, {\parQ*(1 - \THETA) + \THETA/2}) circle [radius = .3pt];
        \node [left] at (-0.1, {\parQ*(1 - \THETA) + \THETA/2}) {$\frac{1}{[q,2]_\theta}$};

        \draw [fill=black] (-0.1,{\parQ*(1 - \THETA) + \THETA}) circle [radius = .3pt];
        \node [left] at (-0.1,{\parQ*(1 - \THETA) + \THETA}) {$\frac{1}{[q,1]_\theta}$};

	\draw [fill=black] (-0.1,1) circle [radius = .3pt];
	\node [left] at (-0.1,1) {$1$};


        \draw [thin,dotted] (0,1) -- (1,1); 
	\draw [thin,dotted] (0,0) -- (1,0); 

        \draw [thin,dotted] (0,{\parQ*(1-\THETA)}) -- (1,{\parQ*(1-\THETA)});
        \draw [thin,dotted] (0,{\parQ*(1 - \THETA) + \THETA/2}) -- (1,{\parQ*(1 - \THETA) + \THETA/2});
        \draw [thin,dotted] (0,{\parQ*(1 - \THETA) + \THETA}) -- (1,{\parQ*(1 - \THETA) + \THETA});


	\draw [thin,dotted] (0,1) -- (0,0); 
	\draw [thin,dotted] (1,1) -- (1,0); 

        \draw [thin,dotted] (\THETA/2,0) -- (\THETA/2,1);
        \draw [thin,dotted] ({\parQ*(1-\THETA) + \THETA/2},0) -- ({\parQ*(1-\THETA) + \THETA/2},1);
        \draw [thin,dotted] ({\parQ*(1 - \THETA) + \THETA},0) -- ({\parQ*(1 - \THETA) + \THETA},1);

	\path [fill=lightgray, opacity = 0.6] (TL)--(P1)--(P2)--(P3)--(P4)--(P5);

        \path [fill=lightgray, opacity = 0.8] (P1)--(P22)--(P32)--(P42);	
\end{tikzpicture}
\end{center}
\label{fig:exponents}
\end{figure}

\begin{proof}
As in the proof of Theorem \ref{thm:HScase}, it suffices to consider decaying multipliers $m\in V_0^s(\Delta;\mc{T})$.
Moreover, by Lemma \ref{lem:embeddingVHR}, Proposition \ref{prop:muckenhoupt}\eqref{it:mw5} and the openness of the upper bound assumptions on $s$, it suffices to consider $m\in R_0^s(\Delta;\mc{T})$.
Throughout the proof we let $r_{s,\theta,q} \in [1,\infty)$ be the unique number such that
  \begin{equation*}
    [q,r_{s,\theta,q}]_\theta = s,
  \end{equation*}
  which exists if $[q, 1]_\theta \leq s < [q, \infty]_\theta$.

  \textbf{Part (i):} First assume $s \neq [q, 1]_\theta$, so that $r_{s,\theta,q}>1$.
  Fix a weight $w \in A_1$.
  Take $t > q$ and define $\sigma  = [t,r_{s,\theta,q}]_\theta>s$.
  By Proposition \ref{prop:main-mult} we have boundedness of the bilinear map
  \begin{equation*}
    R_0^{q}(\Delta;\mc{T}) \times L^{t}(w;Y) \to L^{t}(w;Y), \qquad (m,f) \mapsto T_m f
  \end{equation*}
  using that $\mc{T}$ is $\el{2}{q^\prime}$-bounded on $Y$.
  Moreover, since $s \leq [q,2]_\theta$, we know that $r_{s,\theta,q}\leq 2$, so we have by Theorem \ref{thm:HScase}\eqref{it:HScase1} and Lemma \ref{lem:embeddingVHR} that the bilinear map
  \begin{equation*}
    R_0^{r_{s,\theta,q}}(\Delta;\mc{T}) \times L^{r_{s,\theta,q}}(w;H) \to L^{r_{s,\theta,q}}(w;H), \qquad (m,f) \mapsto T_m f
  \end{equation*}
  is bounded, using
  \begin{equation}\label{eq:unif-bdd-hilbert}
    \nrm{m}_{R^{s}(\Delta;\mc{L}_b(H))} \lesssim \nrm{m}_{R^{s}(\Delta;\mc{T})}
  \end{equation}
  by the uniform boundedness of $\mc{T}$ on $H$.

  We define a bilinear map
  \begin{align*}
    \big( R_0^s(\Delta;\mc{T}) \cap R_0^{r_{s,\theta,q}}(\Delta; \mc{T})\big ) \times &\big(L^t(w;Y) \cap L^{r_{s,\theta,q}}(w;H) \big) \\ &\to L^r(w;Y) \cap L^{r_{s,\theta,q}}(w;H) , \qquad (m,f) \mapsto T_m f.
  \end{align*}
  This is well-defined as it is the extension of the map $(m,f) \mapsto T_m f$ defined for $m \in R_0^{s\wedge r_{s,\theta,q}}(\Delta;\mc{T})$ and $f \in \mc{S}(\RR;Y \cap H)$.
  Here we use that $Y \cap H$ is dense in both $Y$ and $H$.
  By bilinear complex interpolation \cite[\textsection 4.4]{BL76} we have boundedness of
  \begin{align*}
    [R_0^{q}(\Delta;\mc{T}), R_0^{r_{s,\theta,q}}(\Delta;\mc{T})]_\theta \times &[L^{t}(w;Y), L^{r_{r,\theta,q}}(w;H)]_\theta \\
                                                                                &\to [L^{t}(w;Y),L^{r_{s,\theta,q}}(w;H)]_\theta , \qquad (m,f) \mapsto T_m f.
  \end{align*}
  Here we use that the Minkowski norm on the linear span of $\mc{T}$ is complete, i.e. that $\mc{T} \subset \mc{L}_b(Y \cap H)$ is a Banach disc.

  By Theorem \ref{thm:VR-interpoln} we have
  \begin{equation*}
    R_0^{[q,r_{s,\theta,q}]_\theta}(\Delta;\mc{T}) \hookrightarrow [R_0^{q}(\Delta;\mc{T}),R_0^{r_{s,\theta,q}}(\Delta;\mc{T})]_\theta.
  \end{equation*}
  Using this embedding and complex interpolation of weighted Bochner spaces (see \cite[Theorem 1.18.5]{hT78}; note that the proof simply extends to the case $X_0 \neq X_1$), we get boundedness of	
  \begin{equation*}
    R_0^s(\Delta;\mc{T}) \times L^\sigma(w;X) \to L^{\sigma}(w;X), \qquad (m,f) \mapsto T_m f
  \end{equation*}
  with norm estimate
  \begin{equation*}
    \nrm{ \nrm{T_m f}_X }_{L^\sigma(w)} \leq  \inc_{Y, q, s,t, \sigma,\theta} ([w]_{A_{1}}) \nrm{m}_{R^{s}(\Delta;\mc{T})} \elR{2}{q'}{\mc T}\nrm{ \nrm{f}_X }_{L^\sigma(w)}
  \end{equation*}
  for all $w \in A_1$ and all simple functions $f\colon\RR\to X$.
  By scalar-valued extrapolation (see \cite[Theorems 3.9 and Corollary 3.14]{CMP11}) and density of the simple functions we deduce
  \begin{equation*}
    \nrm{T_m f }_{L^p(w;X)} \leq  \inc_{Y,p,q,s,t, \sigma,\theta} ([w]_{A_{p/\sigma}}) \nrm{m}_{R^{s}(\Delta;\mc{T})} \elR{2}{q'}{\mc T}\nrm{f}_{L^p(w;X)}
  \end{equation*}
  for all $p \in [\sigma,\infty)$ and all $w \in A_{p/\sigma}$.
  Taking $t$ arbitrarily close to $q$ and using Proposition \ref{prop:muckenhoupt}\eqref{it:mw5} proves the case $[q, 1]_\theta \neq s$.

  Next if $[q, 1]_\theta = s$ and $w\in A_{p/s}$, then by Proposition \ref{prop:muckenhoupt}\eqref{it:mw5} we can choose $t\in (s, [q, 2]_\theta)$ such that $w\in A_{p/t}$.
  By the previous case $T_m$ is bounded on $L^p(w;X)$ for all $m\in R^{t}(\Delta;\mc{T})$ and hence also for $m\in R^{s}(\Delta;\mc{T})$, which completes the proof.

  \bigskip

\textbf{Part (ii):} Without loss of generality we may assume that $s > [q,2]_\theta$ by embedding of the $R^s$-spaces and the fact that
\begin{equation*}
  \frac{1}{[q,2]_\theta}>\max\cbraces{ \frac{1}{[q,2]_\theta}-\frac{1}{p}, \frac{1-\theta}{q}, \frac1p -\frac{\theta}{2}}
\end{equation*}
for $p >[q,1]_\theta$.
Note that this implies that $r_{s,\theta,q} > 2$.
We will consider three cases:

\textbf{Case 1:} $p \geq [\infty,2]_\theta$.
Since
 \begin{align*}
   \frac{1}{p \theta} &> \frac{1}{\theta}\has{ \frac{\theta}{2}+\frac{1-\theta}{q}  - \frac{1}{s}}
   =  \frac{1}{2} - \frac{1}{r_{s,\theta,q}}
 \end{align*}
 we can find a $p_1> p\theta \geq 2$ such that $p_1<p$ and $p_1 <  \ha{\frac{1}{2} - \frac{1}{r_{s,\theta,q}}}^{-1}$.
 Therefore we know by Theorem \ref{thm:HScase}\eqref{it:HScase2}, using \eqref{eq:unif-bdd-hilbert}, that the bilinear map
 \begin{equation*}
   R_0^{r_{s,\theta,q}}(\Delta;\mc{T}) \times L^{p_1}(\RR;H) \to L^{p_1}(\RR;H), \qquad (m,f) \mapsto T_m f
 \end{equation*}
 is bounded.
 Since $p < [\infty,p_1]_\theta$ we can find a $p_0 \in (p,\infty)$ such that $p= [p_0,p_1]_\theta$.
 By Proposition \ref{prop:main-mult} we have boundedness of the bilinear map
 \begin{equation*}
   R_0^{q}(\Delta;\mc{T}) \times L^{p_0}(\RR;Y) \to L^{p_0}(\RR;Y), \qquad (m,f) \mapsto T_m f,
 \end{equation*}
 using that $\mc{T}$ is $\el{2}{q^\prime}$-bounded on $Y$.
 We can now finish the proof using bilinear complex interpolation, Theorem \ref{thm:VR-interpoln} and complex interpolation of Bochner spaces as in the first part.

 \textbf{Case 2:} $[q,2]_\theta <p <[\infty,2]$.
 Note that $R_0^{r_{s,\theta,q}}(\Delta;\mc{T})\hookrightarrow L^\infty(\RR;\mc{T})$.
 Therefore by Plancherel's theorem and \eqref{eq:unif-bdd-hilbert} the bilinear map
\begin{equation*}
           R_0^{r_{s,\theta,q}}(\Delta;\mc{T}) \times L^{2}(\RR;H) \to L^{2}(\RR;H), \qquad (m,f) \mapsto T_m f
\end{equation*}
is bounded.
Since $[q,2]_\theta <p < [\infty,2]_\theta $ we can find a $p_0 \in (q,\infty)$ such that  $p= [p_0,2]_\theta$.
By Proposition \ref{prop:main-mult} we have boundedness of the bilinear map
\begin{equation*}
  R_0^{q}(\Delta;\mc{T}) \times L^{p_0}(\RR;Y) \to L^{p_0}(\RR;Y), \qquad (m,f) \mapsto T_m f,
\end{equation*}
using that $\mc{T}$ is $\el{2}{q^\prime}$-bounded on $Y$.
The proof can now be finished as before.

\textbf{Case 3:} $[q,1]_\theta <p \leq [q,2]$.
Let $\tilde{p} \in (1,2]$ be such that $p = [q, \tilde{p}]_\theta$.
Then since
 \begin{align*}
   \frac{1}{\tilde{p}}  < \frac{1}{\theta}\has{\frac{\theta}{2} + \frac{1}{s} - \frac{1-\theta}{q}}=  \frac{1}{2} + \frac{1}{r_{s,\theta,q}}
 \end{align*}
 we can find a $1<p_1 < \tilde{p}$ such that $p_1 >  \ha{\frac{1}{2} + \frac{1}{r_{s,\theta,q}}}^{-1}$.
 Therefore we know by Theorem \ref{thm:HScase}\eqref{it:HScase2}, using \eqref{eq:unif-bdd-hilbert}, that the bilinear map
\begin{equation*}
           R_0^{r_{s,\theta,q}}(\Delta;\mc{T}) \times L^{p_1}(\RR;H) \to L^{p_1}(\RR;H), \qquad (m,f) \mapsto T_m f
\end{equation*}
is bounded.
Since $p_1 < \tilde{p}$ we can find a $p_0 \in (q,\infty)$ such that $p= [p_0,p_1]_\theta$. By Proposition \ref{prop:main-mult} we have boundedness of the bilinear map
\begin{equation*}
           R_0^{q}(\Delta;\mc{T}) \times L^{p_0}(\RR;Y) \to L^{p_0}(\RR;Y), \qquad (m,f) \mapsto T_m f,
\end{equation*}
again using that $\mc{T}$ is $\el{2}{q^\prime}$-bounded on $Y$.
The proof can again be finished as before.
\end{proof}

The conditions on $m$ in Theorem \ref{thm:mult-interp-edit}\eqref{it:interp2} with $q=2$ are less restrictive than the conditions of \cite[Theorem 3.6]{HP06}, which allows for Banach spaces with the $\LPR_{p}$ property. The proof of Theorem \ref{thm:mult-interp-edit}\eqref{it:interp2} can also be used to improve the conditions of \cite[Theorem 3.6]{HP06}

\begin{rmk}
A weighted variant of part \eqref{it:interp2} of Theorem \ref{thm:mult-interp-edit} holds for an appropriate class of weights, by using a weighted variant of Theorem \ref{thm:HScase}\eqref{it:HScase2} (see \cite[Theorem A(ii)]{sK14}) and limited range extrapolation (see \cite[Theorem 3.31]{CMP11}). However this involves a reverse H\"older assumption on the weight or the dual weight, so the technical details are therefore left to the interested reader.
\end{rmk}

We continue with an application to $X = L^r$ for $s \in [2,\infty)$.
Results for $s \in (1,2)$ have been previously covered by Example \ref{ex:Lr-small-s}.

\begin{example}\label{ex:Lr-large-s}
  Let $(\Omega, \mu)$ be a measure space and let $p,r \in (1,\infty)$. Let $\mc{T} \subset \mc{L}(\Sigma(\Omega),L^0(\Omega))$ be absolutely convex and $\ell^{2}$-bounded on $L^t(\Omega)$ for all $t\in (1, \infty)$.
  Let $s\in [2, \infty)$ and assume $m \in V^s(\Delta;\mc{T})$.
  Then $T_m$ is bounded on $L^p(\RR;L^r(\Omega))$ in each of the following cases:
  \begin{enumerate}[(i)]
  \item $r\in [2, \infty)$ and $\frac{1}{s}>\max\big\{\frac12-\frac1p,\frac12-\frac1r, \frac1p-\frac1r\}$.
  \item $r\in (1, 2]$ and $\frac{1}{s}>\max\big\{\frac1p-\frac12,\frac1r-\frac12, \frac1r-\frac1p\}$.
  \end{enumerate}
\end{example}

\begin{proof}
  It suffices to prove (i), as (ii) follows from a duality argument.
  Let $\overline{\mc{T}}$ be the closure of $\mc{T}$ in $\calL_b(L^2(\Omega))$.
  Then $\overline{\mc{T}}$ is a Banach disc.
  Moreover, by Lemma \ref{lem:l2-bdd} we know that $ \overline{\mc{T}} \subseteq \calL_b(L^t(\Omega))$ is $\ell^2$-bounded for all $t\in (1, \infty)$.
  We will check the conditions of Theorem \ref{thm:mult-interp-edit}\eqref{it:interp2}  with $\overline{\mc{T}}$, $q=2$, $Y = L^t(\Omega)$ for an appropriate $t>r$ and $H = L^2(\Omega)$. Choose $\theta\in (0,\frac2r)$ such that
\[  \frac{1}{s} > \max\cbraces{\frac{1}{2}-\frac{1}{p},\frac{1-\theta}{2}, \frac1p -\frac{\theta}{2}}.\]
Since $s\geq 2$ it follows that $p>[2,1]_{\theta}$. Now the result follows by choosing $t>r$ such that $r=[t, 2]_{\theta}$.
\end{proof}

In a similar way we obtain the following from Theorem \ref{thm:mult-interp-edit}\eqref{it:interp1} and duality. This partly improves Example \ref{ex:Lr-small-s}.

\begin{example}\label{ex:Lr-sweight}
Let $(\Omega, \mu)$ be a measure space and let $p,r \in (1,\infty)$. Let $\mc{T} \subset \mc{L}(\Sigma(\Omega),L^0(\Omega))$ be absolutely convex and $\ell^{2}$-bounded on $L^t(\Omega)$ for all $t\in [2, \infty)$. Let $s\in (1,2)$ and assume $m \in V^s(\Delta;\mc{T})$.
Then $T_m$ is bounded on $L^p(w;L^r(\Omega))$ if $\frac1p <\frac1s \leq \frac1r+\frac12$ and $w\in A_{p/s}$.
\end{example}

\subsection{Multipliers in intermediate UMD Banach spaces}\label{ssec:int-banach}

In this section we consider general $\UMD$ Banach spaces (not just Banach function spaces) and use interpolation to improve the conditions of Theorem \ref{thm:main-multgeneralUMD} considerably, assuming $X$ is an interpolation space between a $\UMD$ space and a Hilbert space, and using the same interpolation scheme as in Theorem \ref{thm:mult-interp-edit}.
This result is new even for scalar-valued multipliers, and it implies sufficient conditions for Fourier multipliers on the space of Schatten class operators.

\begin{thm}\label{thmintermediateUMD}
  Let $p \in (1,\infty)$ and $\theta \in (0,1)$.
  Let $Y$ and $H$ be an interpolation couple, with $Y \in \UMD$, $H$ a Hilbert space, and $Y \cap H$ dense in both $Y$ and $H$.
  Let $X = [Y,H]_\theta$.
  Suppose $\mc{T} \subset \mc{L}_b(Y \cap H)$ is a Banach disc which is $\mc{R}$-bounded on $Y$ and uniformly bounded on $H$.
  Let $s \in (1,\infty)$ and suppose that $m \in V^s(\Delta;\mc{T})$.

  \begin{enumerate}[(i)]
  \item \label{it:interp1intermed} If $1/s > \min \{ 1/p,1 - (\theta/2) \}$,
    then
    \begin{equation*}
      \|T_m\|_{\mc{L}(L^p(w;X))}\leq \inc_{Y,p, s,  \theta}([w]_{A_{p/s}}) \|m\|_{V^{s}(\Delta;\mc{T})} \elRr{\mc{T}}
    \end{equation*}
    for all $w \in A_{p/s}$.
  \item \label{it:interp2intermed} If
    \begin{equation*}
      \frac{1}{s} > \max\cbraces{ 1-\frac{\theta}{2}-\frac{1}{p}, 1-\theta, \frac1p -\frac{\theta}{2}},
    \end{equation*}
    then
    \begin{equation*}
      \|T_m\|_{\mc{L}(L^p(\RR;X))}\lesssim_{Y, p, s, \theta}  \|m\|_{V^{s}(\Delta;\mc{T})} \elRr{\mc{T}}.
    \end{equation*}
  \end{enumerate}
\end{thm}

The allowable exponents $(p,s)$ above are shown in Figure \ref{fig:exponents2}.

\begin{proof}
To prove the result one can argue as in Theorem \ref{thm:mult-interp-edit} with $q=1$, and using Theorem \ref{thm:main-multgeneralUMD} instead of Proposition \ref{prop:main-mult}.
\end{proof}

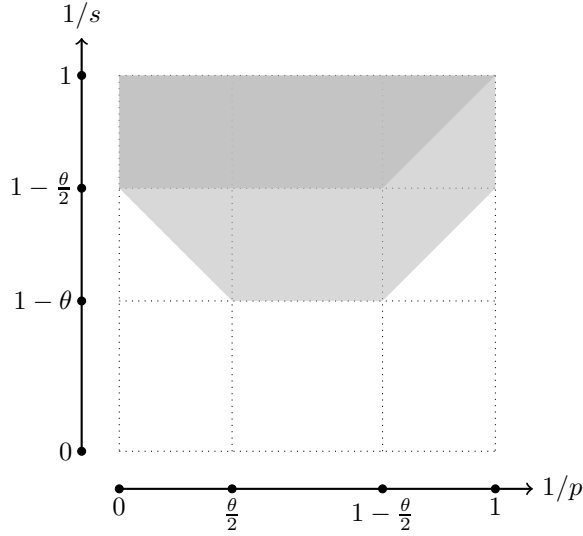
\begin{figure}[h!]
\caption{Allowable exponents for Theorem \ref{thmintermediateUMD}: the weighted case (i) dark shaded, the unweighted case (ii) light shaded.}
\begin{center}
\begin{tikzpicture}[scale=5]

        \pgfmathsetmacro{\THETA}{0.6} 
        \pgfmathsetmacro{\parQ}{(1)} 

	\coordinate (P1) at (0,{\parQ*(1 - \THETA) + \THETA/2});
	\coordinate (P2) at (\THETA/2,{\parQ*(1-\THETA)});
	\coordinate (P3) at ({\parQ*(1 - \THETA) + \THETA/2},{\parQ*(1-\THETA)});
	\coordinate (P4) at ({\parQ*(1 - \THETA) + \THETA},{\parQ*(1 - \THETA) + \THETA/2});
    \coordinate (P5) at ({\parQ*(1 - \THETA) + \THETA},1);

    \coordinate (P22) at ({\parQ*(1 - \THETA) + \THETA/2},{\parQ*(1 - \THETA) + \THETA/2});
    \coordinate (P32) at ({\parQ*(1 - \THETA) + \THETA},{\parQ*(1 - \THETA) + \THETA});
    \coordinate (P42) at (0,{\parQ*(1 - \THETA) + \THETA});

	\coordinate (TL) at (0,1);
	\coordinate (TR) at (1,1);

	\draw [thick,->] (0,-0.1) -- (1.1,-0.1);
        \node [right] at (1.1,-0.1) {$1/p$};

	\draw [fill=black] (0,-0.1) circle [radius = .3pt];
	\node [below] at (0,-0.1) {$0$};

	\draw [fill=black] (1,-0.1) circle [radius = .3pt];
	\node [below] at (1,-0.1) {$1$};

        \draw [fill=black] (\THETA/2, -0.1) circle [radius = .3pt];
        \node [below] at (\THETA/2,-0.1) {$\frac{\theta}{2}$};

        \draw [fill=black] ({\parQ*(1 - \THETA) + \THETA/2},-0.1) circle [radius = .3pt];
        \node [below] at ({\parQ*(1 - \THETA) + \THETA/2},-0.1) {$1-\frac{\theta}{2}$};

	\draw [thick,->] (-0.1,0) -- (-0.1,1.1);
	\node [above] at (-0.1,1.1) {$1/s$};

        \draw [fill=black] (-0.1,0) circle [radius = .3pt];
	\node [left] at (-0.1,0) {$0$};

        \draw [fill=black] (-0.1,{\parQ*(1 - \THETA)}) circle [radius = .3pt];
        \node [left] at (-0.1,{\parQ*(1 - \THETA)}) {$1-\theta$};

        \draw [fill=black] (-0.1, {\parQ*(1 - \THETA) + \THETA/2}) circle [radius = .3pt];
        \node [left] at (-0.1, {\parQ*(1 - \THETA) + \THETA/2}) {$1-\frac{\theta}{2}$};

	\draw [fill=black] (-0.1,1) circle [radius = .3pt];
	\node [left] at (-0.1,1) {$1$};


        \draw [thin,dotted] (0,1) -- (1,1); 
	\draw [thin,dotted] (0,0) -- (1,0); 

        \draw [thin,dotted] (0,{\parQ*(1-\THETA)}) -- (1,{\parQ*(1-\THETA)});
        \draw [thin,dotted] (0,{\parQ*(1 - \THETA) + \THETA/2}) -- (1,{\parQ*(1 - \THETA) + \THETA/2});


	\draw [thin,dotted] (0,1) -- (0,0); 
	\draw [thin,dotted] (1,1) -- (1,0); 

        \draw [thin,dotted] (\THETA/2,0) -- (\THETA/2,1);
        \draw [thin,dotted] ({\parQ*(1-\THETA) + \THETA/2},0) -- ({\parQ*(1-\THETA) + \THETA/2},1);

	\path [fill=lightgray, opacity = 0.6] (TL)--(P1)--(P2)--(P3)--(P4)--(P5);

        \path [fill=lightgray, opacity = 0.8] (P1)--(P22)--(P32)--(P42);	
\end{tikzpicture}
\end{center}
\label{fig:exponents2}
\end{figure}

In the next example we apply Theorem \ref{thmintermediateUMD} to operator-valued multipliers on the Schatten class operators $\Cs^r\subseteq \calL_b(\ell^2)$ for $r\in [1, \infty]$.
This is potentially useful for Schur multipliers (see \cite[Theorem 5.4.3]{HNVW16} and \cite[Theorem 4]{Potapov-Sukochev11}).
For $r\in (1, \infty)$ these spaces have the $\UMD$ property, and for $p,q\in [1, \infty]$ one has $\Cs^{[p,q]_{\theta}} =  [\Cs^{p}, \Cs^{q}]_{\theta}$ (see \cite[Propositions 5.4.2 and D.3.1]{HNVW16}).

\begin{example}
Let $X = \Cs^r$ with $p, r\in (1, \infty)$ and $\mc{T}\subseteq \calL(\Cs^t)$ be absolutely convex and $\mc{R}$-bounded for all $t\in (1, \infty)$.
Let $s\in (1, \infty)$ and assume $m \in V^s(\Delta;\mc{T})$.
Then $T_m$ is bounded on $L^p(\RR;\Cs^r)$ in each of the following cases:
\begin{enumerate}[(i)]
\item $r\in [2, \infty)$ and $\frac{1}{s} > \max\cbraces{ \frac1{p'}-\frac1r, \bigl|\frac1r-\frac{1}{r'}\bigr|, \frac1p -\frac1r}$.
\item $r\in (1, 2]$ and  $\frac{1}{s} > \max\cbraces{ \frac1r - \frac{1}{p'}, \bigl|\frac1r-\frac{1}{r'}\bigr|, \frac1r-\frac1p}$.
\end{enumerate}
In particular, if $p\in [r\wedge r', r\vee r']$ then $T_m$ is bounded on $L^p(\RR;\Cs^r)$ if $r\in (1, \infty)$ and $\frac1s > |\frac1r-\frac{1}{r'}|$.
 \begin{proof}
   The result follows from Theorem \ref{thmintermediateUMD}\eqref{it:interp2intermed} by arguing as in Example \ref{ex:Lr-large-s}. A similar result can be derived on $L^p(w;\Cs^r)$ by Theorem \ref{thmintermediateUMD}\eqref{it:interp1intermed}.
 \end{proof}
\end{example}

\bibliographystyle{plain}
\bibliography{multi}

\begin{thebibliography}{10}

\bibitem{ALV1}
A.~Amenta, E.~Lorist, and M.~Veraar.
\newblock Rescaled extrapolation for vector-valued functions.
\newblock arXiv:1703.06044, 2017.

\bibitem{AreBu02}
W.~Arendt and S.~Bu.
\newblock The operator-valued {M}arcinkiewicz multiplier theorem and maximal
  regularity.
\newblock {\em Math. Z.}, 240(2):311--343, 2002.

\bibitem{BL76}
J.~Bergh and J.~L{\"o}fstr{\"o}m.
\newblock {\em Interpolation spaces. {A}n introduction}.
\newblock Springer-Verlag, Berlin-New York, 1976.
\newblock Grundlehren der Mathematischen Wissenschaften, No. 223.

\bibitem{BerGil94}
E.~Berkson and T.~A. Gillespie.
\newblock Spectral decompositions and harmonic analysis on {UMD} spaces.
\newblock {\em Studia Math.}, 112(1):13--49, 1994.

\bibitem{BGT03}
E.~Berkson, T.~A. Gillespie, and J.~L. Torrea.
\newblock Vector valued transference.
\newblock In P.~Liu, editor, {\em Functional Space Theory and its applications
  (Wuhan, 2003)}, pages 1--27, 2003.

\bibitem{Bour83}
J.~Bourgain.
\newblock Some remarks on {B}anach spaces in which martingale difference
  sequences are unconditional.
\newblock {\em Ark. Mat.}, 21(2):163--168, 1983.

\bibitem{Bou86}
J.~Bourgain.
\newblock Vector-valued singular integrals and the {$H^1$}-{BMO} duality.
\newblock In {\em Probability theory and harmonic analysis ({C}leveland,
  {O}hio, 1983)}, volume~98 of {\em Monogr. Textbooks Pure Appl. Math.}, pages
  1--19. Dekker, New York, 1986.

\bibitem{Buh65}
A.~V. Buhvalov.
\newblock The duality of functors that are generated by spaces of vector-valued
  functions.
\newblock {\em Izv. Akad. Nauk SSSR Ser. Mat.}, 39(6):1284--1309, 1438, 1975.

\bibitem{Burk83}
D.~L. Burkholder.
\newblock A geometric condition that implies the existence of certain singular
  integrals of {B}anach-space-valued functions.
\newblock In {\em Conference on harmonic analysis in honor of Antoni Zygmund,
  Vol. I, II (Chicago, Ill., 1981)}, Wadsworth Math. Ser., pages 270--286.
  Wadsworth, Belmont, CA, 1983.

\bibitem{Burk01}
D.~L. Burkholder.
\newblock Martingales and singular integrals in {B}anach spaces.
\newblock In {\em Handbook of the geometry of {B}anach spaces, {V}ol. {I}},
  pages 233--269. North-Holland, Amsterdam, 2001.

\bibitem{aC64}
A.-P. Calder{\'o}n.
\newblock Intermediate spaces and interpolation, the complex method.
\newblock {\em Studia Math.}, 24:113--190, 1964.

\bibitem{CPSW00}
P.~Cl{\'e}ment, B.~de~Pagter, F.~Sukochev, and H.~Witvliet.
\newblock Schauder decompositions and multiplier theorems.
\newblock {\em Studia Math.}, 138(2):135--163, 2000.

\bibitem{ClePru01}
P.~Cl{\'e}ment and J.~Pr{\"u}ss.
\newblock An operator-valued transference principle and maximal regularity on
  vector-valued {$L\sb p$}-spaces.
\newblock In {\em Evolution equations and their applications in physical and
  life sciences ({B}ad {H}errenalb, 1998)}, volume 215 of {\em Lecture Notes in
  Pure and Appl. Math.}, pages 67--87. Dekker, New York, 2001.

\bibitem{CRS88}
R.~Coifman, J.~L. Rubio~de Francia, and S.~Semmes.
\newblock Multiplicateurs de {F}ourier de {$L^p({\bf R})$} et estimations
  quadratiques.
\newblock {\em C. R. Acad. Sci. Paris S\'er. I Math.}, 306(8):351--354, 1988.

\bibitem{CMP11}
D.~V. Cruz-Uribe, J.~M. Martell, and C.~P{\'e}rez.
\newblock {\em Weights, extrapolation and the theory of {R}ubio de {F}rancia},
  volume 215 of {\em Operator Theory: Advances and Applications}.
\newblock Birkh\"auser/Springer Basel AG, Basel, 2011.

\bibitem{dales2016multi}
H.~G. Dales, N.~J. Laustsen, T.~Oikhberg, and V.~G Troitsky.
\newblock Multi-norms and banach lattices.
\newblock {\em Dissertationes Mathematicae (Rozprawy Matematyczne)}, 2016.

\bibitem{FHL17}
S.~Fackler, T.P. Hyt\"onen, and N.~Lindemulder.
\newblock Weighted estimates for operator-valued multipliers.
\newblock Preprint.

\bibitem{GLV15}
C.~Gallarati, E.~Lorist, and M.~C. Veraar.
\newblock On the {$\ell^s$}-boundedness of a family of integral operators.
\newblock {\em Rev. Mat. Iberoam.}, 32(4):1277--1294, 2016.

\bibitem{GilTor04}
T.~A. Gillespie and J.~L. Torrea.
\newblock Transference of a {L}ittlewood-{P}aley-{R}ubio inequality and
  dimension free estimates.
\newblock {\em Rev. Un. Mat. Argentina}, 45(1):1--6 (2005), 2004.

\bibitem{lG09}
L.~Grafakos.
\newblock {\em Modern {F}ourier analysis}, volume 250 of {\em Graduate Texts in
  Mathematics}.
\newblock Springer, New York, second edition, 2009.

\bibitem{HHN02}
R.~Haller, H.~Heck, and A.~Noll.
\newblock Mikhlin's theorem for operator-valued {F}ourier multipliers in {$n$}
  variables.
\newblock {\em Math. Nachr.}, 244:110--130, 2002.

\bibitem{HNVW16}
T.~P. Hyt\"onen, J.~M. A. M.~van Neerven, M.~C. Veraar, and L.~Weis.
\newblock {\em Analysis in {B}anach {S}paces. {V}olume~{I}:~{M}artingales and
  {L}ittlewood-{P}aley Theory}, volume~63 of {\em Ergebnisse der Mathematik und
  ihrer Grenzgebiete. 3. Folge.}
\newblock Springer International Publishing, 2016.

\bibitem{HNVW2}
T.~P. Hyt\"onen, J.~M. A. M.~van Neerven, M.~C. Veraar, and L.~Weis.
\newblock Analysis in {B}anach {S}paces. {V}olume~{II}:~{P}robabilistic
  {M}ethods and {O}perator {T}heory.
\newblock In preparation., 2017.

\bibitem{HP06}
T.~P. Hyt{\"o}nen and D.~Potapov.
\newblock Vector-valued multiplier theorems of {C}oifman-{R}ubio de
  {F}rancia-{S}emmes type.
\newblock {\em Arch. Math. (Basel)}, 87(3):245--254, 2006.

\bibitem{HTY09}
T.~P. Hyt{\"o}nen, J.~L. Torrea, and D.~V. Yakubovich.
\newblock The {L}ittlewood-{P}aley-{R}ubio de {F}rancia property of a {B}anach
  space for the case of equal intervals.
\newblock {\em Proc. Roy. Soc. Edinburgh Sect. A}, 139(4):819--832, 2009.

\bibitem{sK09}
S.~V. Kislyakov.
\newblock The {L}ittlewood-{P}aley theorem for arbitrary intervals: weighted
  estimates.
\newblock {\em J. Math. Sci (N. Y.)}, 156(5):824--833, 2009.

\bibitem{sK14}
S.~Kr{\'o}l.
\newblock Fourier multipliers on weighted {$L^p$} spaces.
\newblock {\em Math. Res. Lett.}, 21(4):807--830, 2014.

\bibitem{KU14}
P.~Kunstmann and A.~Ullmann.
\newblock {$\mathscr{R}_s$}-sectorial operators and generalized
  {T}riebel-{L}izorkin spaces.
\newblock {\em J. Fourier Anal. Appl.}, 20(1):135--185, 2014.

\bibitem{KW04}
P.~Kunstmann and L.~Weis.
\newblock Maximal {$L_p$}-regularity for parabolic equations, {F}ourier
  multiplier theorems and {$H^\infty$}-functional calculus.
\newblock In {\em Functional analytic methods for evolution equations}, volume
  1855 of {\em Lecture Notes in Math.}, pages 65--311. Springer, Berlin, 2004.

\bibitem{Kwapfactor72}
S.~Kwapie\'n.
\newblock On operators factorizable through {$L_{p}$} space.
\newblock In {\em Actes du {C}olloque d'{A}nalyse {F}onctionnelle de {B}ordeaux
  ({U}niv. de {B}ordeaux, 1971)}, pages 215--225. Bull. Soc. Math. France,
  M\'em. No. 31--32. Soc. Math. France, Paris, 1972.

\bibitem{KVW16}
S.~Kwapie{\'n}, M.~C. Veraar, and L.~Weis.
\newblock {$R$}-boundedness versus {$\gamma$}-boundedness.
\newblock {\em Ark. Mat.}, 54(1):125--145, 2016.

\bibitem{mL07}
M.T. Lacey.
\newblock {\em Issues related to {R}ubio de {F}rancia's {L}ittlewood-{P}aley
  inequality}, volume~2 of {\em New York Journal of Mathematics. NYJM
  Monographs}.
\newblock State University of New York, University at Albany, Albany, NY, 2007.

\bibitem{LT91}
M.~Ledoux and M.~Talagrand.
\newblock {\em Probability in {B}anach spaces}, volume~23 of {\em Ergebnisse
  der Mathematik und ihrer Grenzgebiete (3) [Results in Mathematics and Related
  Areas (3)]}.
\newblock Springer-Verlag, Berlin, 1991.
\newblock Isoperimetry and processes.

\bibitem{nL14}
N.~Lindemulder.
\newblock Parabolic initial-boundary value problems with inhomogeneous data.
\newblock Master's thesis, Utrecht University, Utrecht, the Netherlands, 2014.

\bibitem{Lin14}
N.~Lindemulder.
\newblock Banach space-valued extensions of linear operators on {$L^\infty$}.
\newblock In {\em Ordered Structures and Applications (Positivity VII, Zaanen
  Centennial Conference)}. Birkh\"auser Verlag, 2016.

\bibitem{LT79}
J.~Lindenstrauss and L.~Tzafriri.
\newblock {\em Classical {B}anach spaces. {II}}, volume~97 of {\em Ergebnisse
  der Mathematik und ihrer Grenzgebiete}.
\newblock Springer-Verlag, Berlin-New York, 1979.
\newblock Function spaces.

\bibitem{MeyNie}
P.~Meyer-Nieberg.
\newblock {\em Banach lattices}.
\newblock Universitext. Springer-Verlag, Berlin, 1991.

\bibitem{MV15}
M.~Meyries and M.~C. Veraar.
\newblock Pointwise multiplication on vector-valued function spaces with power
  weights.
\newblock {\em J. Fourier Anal. Appl.}, 21(1):95--136, 2015.

\bibitem{sM96}
S.~Montgomery-Smith.
\newblock Stability and dichotomy of positive semigroups on {$L_p$}.
\newblock {\em Proc. Amer. Math. Soc.}, 124(8):2433--2437, 1996.

\bibitem{NVW15}
J.~M. A. M.~van Neerven, M.~C. Veraar, and L.~Weis.
\newblock On the {$R$}-boundedness of stochastic convolution operators.
\newblock {\em Positivity}, 19(2):355--384, 2015.

\bibitem{NW05}
J.~M. A. M.~van Neerven and L.~Weis.
\newblock Stochastic integration of functions with values in a {B}anach space.
\newblock {\em Studia Math.}, 166(2):131--170, 2005.

\bibitem{PCB87}
P.~P\'erez~Carreras and J.~Bonet.
\newblock {\em Barrelled locally convex spaces}, volume 131 of {\em
  North-Holland Mathematics Studies}.
\newblock North-Holland Publishing Co., Amsterdam, 1987.
\newblock Notas de Matem\'atica [Mathematical Notes], 113.

\bibitem{gP16}
G.~Pisier.
\newblock {\em {Martingales in Banach spaces.}}
\newblock Cambridge: Cambridge University Press, 2016.

\bibitem{Potapov-Sukochev11}
D.~Potapov and F.~Sukochev.
\newblock Operator-{L}ipschitz functions in {S}chatten-von {N}eumann classes.
\newblock {\em Acta Math.}, 207(2):375--389, 2011.

\bibitem{PSX12}
D.~Potapov, F.~Sukochev, and Q.~Xu.
\newblock On the vector-valued {L}ittlewood-{P}aley-{R}ubio de {F}rancia
  inequality.
\newblock {\em Rev. Mat. Iberoam.}, 28(3):839--856, 2012.

\bibitem{jR85}
J.~L. Rubio~de Francia.
\newblock A {L}ittlewood-{P}aley inequality for arbitrary intervals.
\newblock {\em Rev. Mat. Iberoam.}, 1(2):1--14, 1985.

\bibitem{jR86}
J.~L. Rubio~de Francia.
\newblock Martingale and integral transforms of {B}anach space valued
  functions.
\newblock In {\em Probability and {B}anach spaces ({Z}aragoza, 1985)}, volume
  1221 of {\em Lecture Notes in Math.}, pages 195--222. Springer, Berlin, 1986.

\bibitem{StrWeis08}
{\v{Z}}.~{\v{S}}trkalj and L.~Weis.
\newblock On operator-valued {F}ourier multiplier theorems.
\newblock {\em Trans. Amer. Math. Soc.}, 359(8):3529--3547 (electronic), 2007.

\bibitem{hT78}
H.~Triebel.
\newblock {\em Interpolation theory, function spaces, differential operators},
  volume~18 of {\em North-Holland Mathematical Library}.
\newblock North-Holland Publishing Co., Amsterdam-New York, 1978.

\bibitem{Weis01a}
L.~Weis.
\newblock A new approach to maximal {$L\sb p$}-regularity.
\newblock In {\em Evolution equations and their applications in physical and
  life sciences ({B}ad {H}errenalb, 1998)}, volume 215 of {\em Lecture Notes in
  Pure and Appl. Math.}, pages 195--214. Dekker, New York, 2001.

\bibitem{We01}
L.~Weis.
\newblock Operator-valued {F}ourier multiplier theorems and maximal
  {$L_p$}-regularity.
\newblock {\em Math. Ann.}, 319(4):735--758, 2001.

\bibitem{Xu96}
Q.~Xu.
\newblock Fourier multipliers for {$L_p({\bf R}^n)$} via {$q$}-variation.
\newblock {\em Pacific J. Math.}, 176(1):287--296, 1996.

\end{thebibliography}

\end{document}